\documentclass[11pt,letterpaper,reqno]{amsart}

\title[Semiexplicit Symplectic Integrators]{Semiexplicit Symplectic Integrators\\for Non-separable Hamiltonian Systems}
\author{Buddhika Jayawardana}
\author{Tomoki Ohsawa}
\address{Department of Mathematical Sciences, The University of Texas at Dallas, 800 W Campbell Rd, Richardson, TX 75080-3021}
\email{Buddhika.Jayawardana@utdallas.edu,tomoki@utdallas.edu}

\date{\today}



\usepackage{graphicx,amssymb,caption,paralist,epstopdf}

\usepackage[margin=1in, marginpar=.5in]{geometry}

\usepackage[numbers,sort&compress]{natbib}
\usepackage[colorlinks=false]{hyperref}

\usepackage[nameinlink]{cleveref}

\usepackage{tikz}
\usetikzlibrary{arrows,positioning}
\usepackage{tikz-cd}

\theoremstyle{plain}
\newtheorem{theorem}{Theorem}
\newtheorem*{theorem*}{Theorem}

\newtheorem{lemma}[theorem]{Lemma}
\newtheorem{proposition}[theorem]{Proposition}

\theoremstyle{definition}
\newtheorem{definition}[theorem]{Definition}

\theoremstyle{remark}

\def\pd#1#2{\dfrac{\partial #1}{\partial #2}}

\def\parentheses#1{{\left(#1\right)}}

\def\braces#1{{\left\{#1\right\}}}

\def\norm#1{{\left\|#1\right\|}}

\def\R{\mathbb{R}}

\def\N{\mathbb{N}}
\def\defeq{\mathrel{\mathop:}=}

\def\setdef#1#2{{\left\{ #1 \ |\ #2 \right\}}}

\def\eps{\varepsilon}

\def\d{{\bf d}}

\def\dt{\Delta t}
\DeclareMathOperator{\sgn}{sgn}

\begin{document}

\footskip=.6in

\begin{abstract}
  We construct a symplectic integrator for non-separable Hamiltonian systems combining an extended phase space approach of Pihajoki and the symmetric projection method. The resulting method is semiexplicit in the sense that the main time evolution step is explicit whereas the symmetric projection step is implicit. The symmetric projection binds potentially diverging copies of solutions, thereby remedying the main drawback of the extended phase space approach. Moreover, our semiexplicit method is symplectic in the original phase space. This is in contrast to existing extended phase space integrators, which are symplectic only in the extended phase space. We demonstrate that our method exhibits an excellent long-time preservation of invariants, and also that it tends to be as fast as and can be faster than Tao's explicit modified extended phase space integrator particularly for small enough time steps and with higher-order implementations and for higher-dimensional problems.
\end{abstract}

\maketitle

\section{Introduction}

\subsection{Non-Separable Hamiltonian Systems}
Consider a Hamiltonian system
\begin{equation}
  \label{eq:Ham}
  \dot{q} = D_{2}H(q, p),
  \qquad
  \dot{p} = -D_{1}H(q, p).
\end{equation}
with Hamiltonian $H\colon T^{*}\R^{d} \to \R$, where $D_{i}$ stands for the partial derivative with respect to the $i^{\rm th}$ set of variables.

Its underlying symplectic structure on the cotangent bundle $T^{*}\R^{d}$ is
\begin{equation}
  \label{eq:Omega}
  \Omega \defeq \d{q} \wedge \d{p} = \sum_{i=1}^{d} \d{q}^{i} \wedge \d{p}_{i}.
\end{equation}
Throughout the paper, we shall use shorthands like above suppressing the summation symbol and indices when writing 2-forms.

The Hamiltonian $H$ and the Hamiltonian system~\eqref{eq:Ham} are said to be \textit{separable} if $H$ can be written as $H(q,p) = K(p)+ V(q)$ with some functions $K$ and $V$, and \textit{non-separable} otherwise.
The separability is a crucial concept for the development of symplectic integrators for Hamiltonian systems because the separability sometimes turns implicit methods into explicit ones, and also makes the splitting method more amenable due to the fact that those flows of \eqref{eq:Ham} with $H(q,p) = K(p)$ and $H(q,p) = V(q)$ are both exactly solvable; see, e.g., \citet{SaCa2018}, \citet{LeRe2004}, and \citet{HaLuWa2006}.

On the other hand, developing an efficient symplectic integrator for non-separable Hamiltonian systems is a challenge.
Although one can obtain explicit integrators for certain classes of non-separable Hamiltonian systems~\cite{StLa2000,Bl2002,WuFoRo2003,McQu2004,Ch2009,Ta2016a,WaSuLiWu2021a,WaSuLiWu2021b,WaSuLiWu2021c,WuWaSuLi2021}, these methods are specialized for their respective specific forms of Hamiltonians, and do not seem to generalize to arbitrary non-separable Hamiltonians in a simple manner.

\subsection{Extended Phase Space Approach}
\label{ssec:extend_phase_space}
To our knowledge, the first explicit integrator for general non-separable Hamiltonian systems was developed by \citet{Pi2015}.
Specifically, in order to find an explicit integrator for non-separable Hamiltonian systems~\eqref{eq:Ham}, \citeauthor{Pi2015} proposed to solve the following extended system instead:
\begin{equation}
  \label{eq:Ham-extended}
  \begin{aligned}
    \dot{q} &= D_{2}H(x, p), \qquad & \dot{p} &= -D_{1}H(q, y), \medskip\\
    \dot{x} &= D_{2}H(q, y), \qquad & \dot{y} &= -D_{1}H(x, p).
  \end{aligned}
\end{equation}

One sees that if we impose the initial condition
\begin{equation*}
  (q(0), x(0), p(0), y(0)) = (q_{0}, q_{0}, p_{0}, p_{0}),
\end{equation*}
then the solution satisfies $(q(t), p(t)) = (x(t), y(t))$ for any $t \in \R$ (assuming that the solution exists and is unique), and $t \mapsto (q(t),p(t))$ coincides with the solution of the original system~\eqref{eq:Ham} with $(q(0), p(0)) = (q_{0}, p_{0})$.
In other words, the subspace
\begin{equation}
  \label{eq:mathcalN}
  \mathcal{N} \defeq \setdef{ (q, q, p, p) \in T^{*}\R^{2d} }{ (q, p) \in T^{*}\R^{d} } \subset T^{*}\R^{2d}
\end{equation}
is an invariant submanifold of \eqref{eq:Ham-extended}, and the system~\eqref{eq:Ham-extended} restricted to this submanifold gives two copies of the original system~\eqref{eq:Ham}.

Notice also that \eqref{eq:Ham-extended} is a Hamiltonian system defined on the extended phase space
\begin{equation*}
  T^{*}\R^{2d} = \setdef{ (q, x, p, y) \in \R^{4d} }{ (q, x) \in \R^{2d},\ (p, y) \in T_{(q,x)}^{*}\R^{2d} \cong \R^{2d} }.
\end{equation*}
More specifically, the extended system \eqref{eq:Ham-extended} is a Hamiltonian system on $T^{*}\R^{2d}$ with the extended Hamiltonian
\begin{equation*}
  \hat{H}\colon T^{*}\R^{2d} \to \R;
  \qquad
  \hat{H}(q, x, p, y) \defeq H(q, y) + H(x, p)
\end{equation*}
and the following standard symplectic form on $T^{*}\R^{2d}$:
\begin{equation}
  \label{eq:hatOmega}
  \hat{\Omega} \defeq \d{q} \wedge \d{p} + \d{x} \wedge \d{y}.
\end{equation}

Let us give a more geometric and intuitive interpretation; see also \Cref{fig:extended_system}.
\begin{figure}[hbtp]
  \centering
  \includegraphics[width=.6\linewidth]{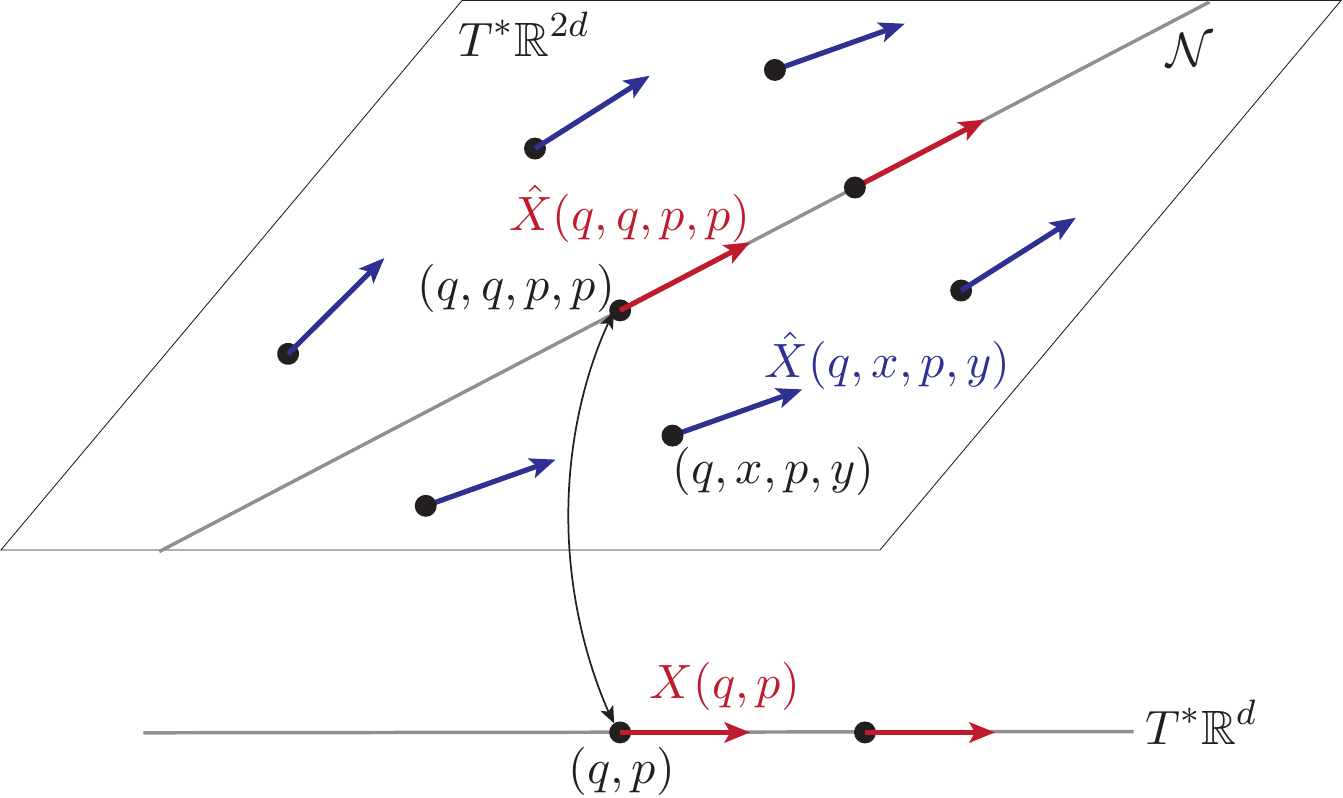}
  \caption{Hamiltonian vector field $\hat{X}$ on extended phase space $T^{*}\R^{2d}$ restricted to $\mathcal{N}$ gives two copies of Hamiltonian vector field $X$ on $T^{*}\R^{d}$.}
  \label{fig:extended_system}
\end{figure}
Let $\hat{X}$ be the vector field on $T^{*}\R^{2d}$ for the extended system~\eqref{eq:Ham-extended}:
\begin{equation*}
  \hat{X}(q, x, p, y) \defeq \parentheses{
    D_{2}H(x, p),\,
    D_{2}H(q, y),\,
    -D_{1}H(q, y),\,
    -D_{1}H(x, p)
  }.
\end{equation*}
If we restrict $\hat{X}$ to the submanifold $\mathcal{N}$, then $\hat{X}(q, q, p, p)$ consists of two copies of the Hamiltonian vector field 
\begin{equation*}
  X(q, p) \defeq \parentheses{
    D_{2}H(q, p),\,
    -D_{1}H(q, p)
  }
\end{equation*}
on the original phase space $T^{*}\R^{d}$.
Therefore, the dynamics in $\mathcal{N}$ defined by $\hat{X}$ is two copies of the Hamiltonian dynamics defined by $X$ on $T^{*}\R^{d}$.
Hence they are essentially the same dynamics.

\subsection{Extended Phase Space Integrators}
\label{ssec:extend_phase_space_integrators}
The salient feature of the above approach by \citet{Pi2015} is that the Hamiltonian is now separable:
Defining $\hat{H}_{A}, \hat{H}_{B} \colon T^{*}\R^{2d} \to \R$ by setting
\begin{equation*}
  \hat{H}_{A}(q, x, p, y) \defeq H(q, y),
  \qquad
  \hat{H}_{B}(q, x, p, y) \defeq H(x, p),
\end{equation*}
we have
\begin{equation*}
  \hat{H}(q, x, p, y)
  = \hat{H}_{A}(q, x, p, y) + \hat{H}_{B}(q, x, p, y)
  = H(q, y) + H(x, p).
\end{equation*}
Furthermore, the flows $\hat{\Phi}^{A}_{t}, \hat{\Phi}^{B}_{t}\colon T^{*}\R^{2d} \to T^{*}\R^{2d}$ corresponding to the Hamiltonians $\hat{H}_{A}$ and $\hat{H}_{B}$ are exactly solvable for any $t \in \R$.
Indeed, with $\hat{H}_{A}$, we have
\begin{equation*}
  \begin{aligned}
    \dot{q} &= 0, \qquad & \dot{p} &= -D_{1}H(q, y), \medskip\\
    \dot{x} &= D_{2}H(q, y), \qquad  & \dot{y} &= 0.
  \end{aligned}
\end{equation*}
and so
\begin{equation*}
  \hat{\Phi}^{A}_{t}(q_{0}, x_{0}, p_{0}, y_{0}) =
  \parentheses{
    q_{0},\,
    x_{0} + t\,D_{2}H(q_{0}, y_{0}),\,
    p_{0} - t\,D_{1}H(q_{0}, y_{0}),\,
    y_{0}
  },
\end{equation*}
whereas, with $\hat{H}_{B}$, we have
\begin{equation*}
  \begin{aligned}
    \dot{q} &= D_{2}H(x, p), \qquad & \dot{p} &= 0, \medskip\\
    \dot{x} &= 0, \qquad & \dot{y} &= -D_{1}H(x, p).
  \end{aligned}
\end{equation*}
and so
\begin{equation*}
  \hat{\Phi}^{B}_{t}(q_{0}, x_{0}, p_{0}, y_{0}) =
  \parentheses{
    q_{0} + t\,D_{2}H(x_{0}, p_{0}),\,
    x_{0},\,
    p_{0},\,
    y_{0} - t\,D_{1}H(x_{0}, p_{0})
  }.
\end{equation*}

One can therefore construct a $2^{\rm nd}$-order explicit integrator by using the Strang splitting~\cite{St1968}:
\begin{equation}
  \label{eq:Strang}
  \hat{\Phi}_{\dt} \defeq \hat{\Phi}^{A}_{\dt/2} \circ \hat{\Phi}^{B}_{\dt} \circ \hat{\Phi}^{A}_{\dt/2}
  \colon T^{*}\R^{2d} \to T^{*}\R^{2d}
\end{equation}
with time step $\dt$.

\subsection{Issues of Extended Phase Space Integrators}
\label{ssec:issues}
As ingenious as the above idea is, there are some issues to be addressed for such extended phase space integrators.

Perhaps the most significant issue is that the defect $(x - q, y - p)$ in the phase space copies $(q,p)$ and $(x,y)$ tends to grow in time numerically.
This can be detrimental in the long run because the solution of the extended system~\eqref{eq:Ham-extended} may diverge from that of the original system~\eqref{eq:Ham} when $(x, y) \neq (q, p)$.
We shall demonstrate it in \Cref{fig:Schrodinger_diff} in \Cref{ssec:invariants} below.

As a remedy for this drawback, \citet{Ta2016b} suggested to solve an alternative extended system
\begin{equation}
  \label{eq:Tao}
  \begin{aligned}
    \dot{q} &= D_{2}H(x, p) + \omega(p - y), \qquad & \dot{p} &= -D_{1}H(q, y) - \omega(q - x), \medskip\\
    \dot{x} &= D_{2}H(q, y) + \omega(y - p), \qquad  & \dot{y} &= -D_{1}H(x, p) - \omega(x - q)
  \end{aligned}
\end{equation}
with some $\omega \in \R$.
Note that the above system~\eqref{eq:Tao} is also a Hamiltonian system on the extended phase space $T^{*}\R^{2d}$ with the Hamiltonian
\begin{equation*}
  \hat{H}_{A}(q, x, p, y) + \hat{H}_{B}(q, x, p, y) + \hat{H}_{C}(q, x, p, y),
\end{equation*}
where
\begin{equation*}
  \hat{H}_{C}(q, x, p, y) \defeq \frac{\omega}{2}\parentheses{ (x - q)^{2} + (y - p)^{2} }.
\end{equation*}
Also, the submanifold $\mathcal{N} \subset T^{*}\R^{2d}$ (defined in \eqref{eq:mathcalN}) is again an invariant manifold of \eqref{eq:Tao} as well.

The flow corresponding to $\hat{H}_{C}$ is also exactly solvable, and so one can construct symplectic integrators using splitting methods based on the three flows corresponding to the three Hamiltonians.
For example, \citet{Ta2016b} constructed the following $2^{\rm nd}$-order integrator:
\begin{equation}
  \label{eq:Tao-Strang}
  \hat{\Phi}^{A}_{\dt/2} \circ \hat{\Phi}^{B}_{\dt/2} \circ \hat{\Phi}^{C}_{\dt} \circ \hat{\Phi}^{B}_{\dt/2} \circ \hat{\Phi}^{A}_{\dt/2}
  \colon T^{*}\R^{2d} \to T^{*}\R^{2d}.
\end{equation}
It is demonstrated numerically in \cite{Ta2016b} (see also \Cref{fig:Schrodinger_diff} in \Cref{ssec:invariants} below) that the insertion of the $C$-step effectively suppresses the growth of the defect $(x - q, y - p)$.
We note however that the defect is far from negligible as we shall address later.

We also mention in passing that various modifications of extended phase space integrators are proposed, especially for relativistic dynamics with astrophysical applications~\cite{LiWu2017,LiWuHuLi2016,LuWuHuLi2017,PhysRevD.104.044055,ZhZhLiWu2022}. 

Another issue is that, while these integrators are symplectic in the extended phase space $T^{*}\R^{2d}$, it is not clear how that is related to the symplecticity of the original system~\eqref{eq:Ham} in $T^{*}\R^{d}$.
More specifically, given a discrete flow $\{ (q_{n}, x_{n}, p_{n}, y_{n}) \}_{n \ge 0}$ on $T^{*}\R^{2d}$ constructed by one of the above extended phase space integrators, one easily sees that the extended symplectic form $\hat{\Omega}$ defined in \eqref{eq:hatOmega} is preserved, but it is not clear if the projected dynamics $\{ (q_{n}, p_{n}) \}_{n \ge 0}$ on $T^{*}\R^{d}$ preserves the symplectic form $\Omega$ defined in \eqref{eq:Omega}.
In fact, such a projected dynamics on $T^{*}\R^{d}$ is not well-defined because the time evolution $(q_{n}, p_{n}) \mapsto (q_{n+1}, p_{n+1})$ depends on $(x_{n}, y_{n})$ as well.

\subsection{Main Result and Outline}
We propose to combine the above extended phase space approach of \citet{Pi2015} with the symmetric projection method (see, e.g., \cite[Section~V.4.1]{HaLuWa2006}) to construct a symplectic integrator for non-separable Hamiltonian systems~\eqref{eq:Ham}.

The resulting method is semiexplicit in the sense that the main time evolution step is explicit whereas the symmetric projection step is implicit.
We will describe the details of the method in \Cref{ssec:method} below, but let us give an overview of the method here.
Given an initial point $(q_{0}, p_{0})$, we construct discrete dynamics $\{ z_{n} = (q_{n}, p_{n}) \}_{n \ge 1}$ as follows:
For any known $z_{n} = (q_{n}, p_{n})$, we define $z_{n+1} = (q_{n+1}, p_{n+1})$ following these steps:
\begin{itemize}
\item $\zeta_{n} \defeq (q_{n}, q_{n}, p_{n}, p_{n}) \in T^{*}\R^{2d}$;
  \smallskip
\item $\hat{\zeta}_{n} = (\hat{q}_{n}, \hat{x}_{n}, \hat{p}_{n}, \hat{y}_{n}) \defeq \zeta_{n} + \xi_{n}$;
  \smallskip
\item $\hat{\zeta}_{n+1} \defeq \hat{\Phi}_{\dt}(\hat{\zeta}_{n})$ using an explicit extended phase space integrator $\hat{\Phi}_{\dt}$ such as \eqref{eq:Strang};
  \smallskip
\item $\zeta_{n+1} = (q_{n+1}, q_{n+1}, p_{n+1}, p_{n+1}) \defeq \hat{\zeta}_{n+1} + \xi_{n} \in \mathcal{N}$;
  \smallskip
\item $z_{n+1} \defeq (q_{n+1}, p_{n+1})$,
\end{itemize}
where $\xi_{n} \in T^{*}\R^{2d}$ is a function of $z_{n}$ and $\dt$ determined by solving nonlinear equations so that $\hat{\zeta}_{n+1}$ is projected to $\zeta_{n+1}$ in $\mathcal{N}$; this is the main part of the symmetric projection step and is implicit.

The proposed method resolves those issues of the extended phase space integrators raised in the previous subsection.
First, the projection step ensures $(x_{n}, y_{n}) = (q_{n}, p_{n})$ for any $n \ge 0$---eliminating the problematic defect $(x_{n} - q_{n}, y_{n} - p_{n})$.
As we shall show in \Cref{ssec:existence}, this also implies that our method defines a discrete flow $\Phi_{\dt}\colon z_{n} \mapsto z_{n+1}$ in the \textit{original} phase space $T^{*}\R^{d}$.
We will also show in \Cref{ssec:symmetry} that our method is symmetric assuming that $\hat{\Phi}$ is symmetric.

Our main result is that the resulting method $\Phi$ is symplectic in the \textit{original} phase space $T^{*}\R^{d}$ given that the extended phase space integrator $\hat{\Phi}$ is symplectic in the extended one $T^{*}\R^{2d}$.
We will prove this in \Cref{sec:symplecticity} along with a geometric interpretation of the integrator.

Finally, in \Cref{sec:results}, we will show the implementation of the method as well as some numerical results.
We will demonstrate that our method has desired orders of accuracy as well as that it preserves invariants in long-time simulations.
Since the Newton-type iterations required in the projection step is very simple, the method tends to be much faster than implicit symplectic methods of the same order.
The symmetric projection step being implicit, our method tends to be slower than the explicit method of \citet{Ta2016b} mentioned above for low-dimensional problems.
However, for higher-order methods and higher-dimensional problems, our method becomes as fast as \citeauthor{Ta2016b}'s.
This demonstrates that our method resolves the issues of extended phase space integrators without sacrificing the computational efficiency for practical problems.

\section{Extended Phase Space Integrators with Symmetric Projection}
\subsection{Definition of the Method}
\label{ssec:method}
Let us define a linear map
\begin{equation*}
  A\colon T^{*}\R^{2d} \cong \R^{4d} \to \R^{2d};
  \qquad
  A(q, x, p, y) \defeq (q - x, p - y),
\end{equation*}
or in the matrix representation
\begin{equation}
  \label{eq:A}
  A =
  \begin{bmatrix}
    I_{d} & -I_{d} & 0 & 0 \\
    0 & 0 & I_{d} & -I_{d}
  \end{bmatrix}.
\end{equation}
We see that
\begin{equation*}
  \ker A = \setdef{ (q, q, p, p) \in T^{*}\R^{2d} \cong \R^{4d} }{ q, p \in \R^{d} } = \mathcal{N},
\end{equation*}
where $\mathcal{N}$ is defined in \eqref{eq:mathcalN}.

Then, a more detailed description of our method is the following (see also \Cref{fig:scheme} below):
\begin{definition}[Extended phase space integrator with symmetric projection]
  \label{def:method}
  \leavevmode
  Given an extended phase space integrator $\hat{\Phi}_{\dt}\colon T^{*}\R^{2d} \to T^{*}\R^{2d}$ and $z_{n} = (q_{n}, p_{n}) \in T^{*}\R^{d}$,
  \begin{enumerate}
    \renewcommand{\theenumi}{\arabic{enumi}}
    \renewcommand{\labelenumi}{\sf\theenumi.}
  \item $\zeta_{n} \defeq (q_{n}, q_{n}, p_{n}, p_{n})$
  \item Find $\mu \in \R^{2d}$ such that $\hat{\Phi}_{\dt}(\zeta_{n} + A^{T}\mu) + A^{T} \mu \in \mathcal{N}$
    \label{step:shift1}
  \item $\hat{\zeta}_{n} \defeq \zeta_{n} + A^{T} \mu$
  \item $\hat{\zeta}_{n+1} \defeq \hat{\Phi}_{\dt}(\hat{\zeta}_{n})$
  \item $\zeta_{n+1} = (q_{n+1}, q_{n+1}, p_{n+1}, p_{n+1}) \defeq \hat{\zeta}_{n+1} + A^{T} \mu$
    \label{step:shift2}
  \item $z_{n+1} \defeq (q_{n+1}, p_{n+1})$
  \end{enumerate}
  Note that Steps~\ref{step:shift1}--\ref{step:shift2} combined are equivalent to solving the nonlinear equations
  \begin{equation}
    \label{eq:nonlinear_eqs}
    \zeta_{n+1} = \hat{\Phi}_{\dt}(\zeta_{n} + A^{T}\mu) + A^{T}\mu
    \quad\text{and}\quad
    A \zeta_{n+1} = 0
  \end{equation}
  for $(\mu, \zeta_{n+1}) \in \R^{2d} \times T^{*}\R^{2d}$.
\end{definition}

\begin{figure}[hbtp]
  \centering
  \includegraphics[width=.6\linewidth]{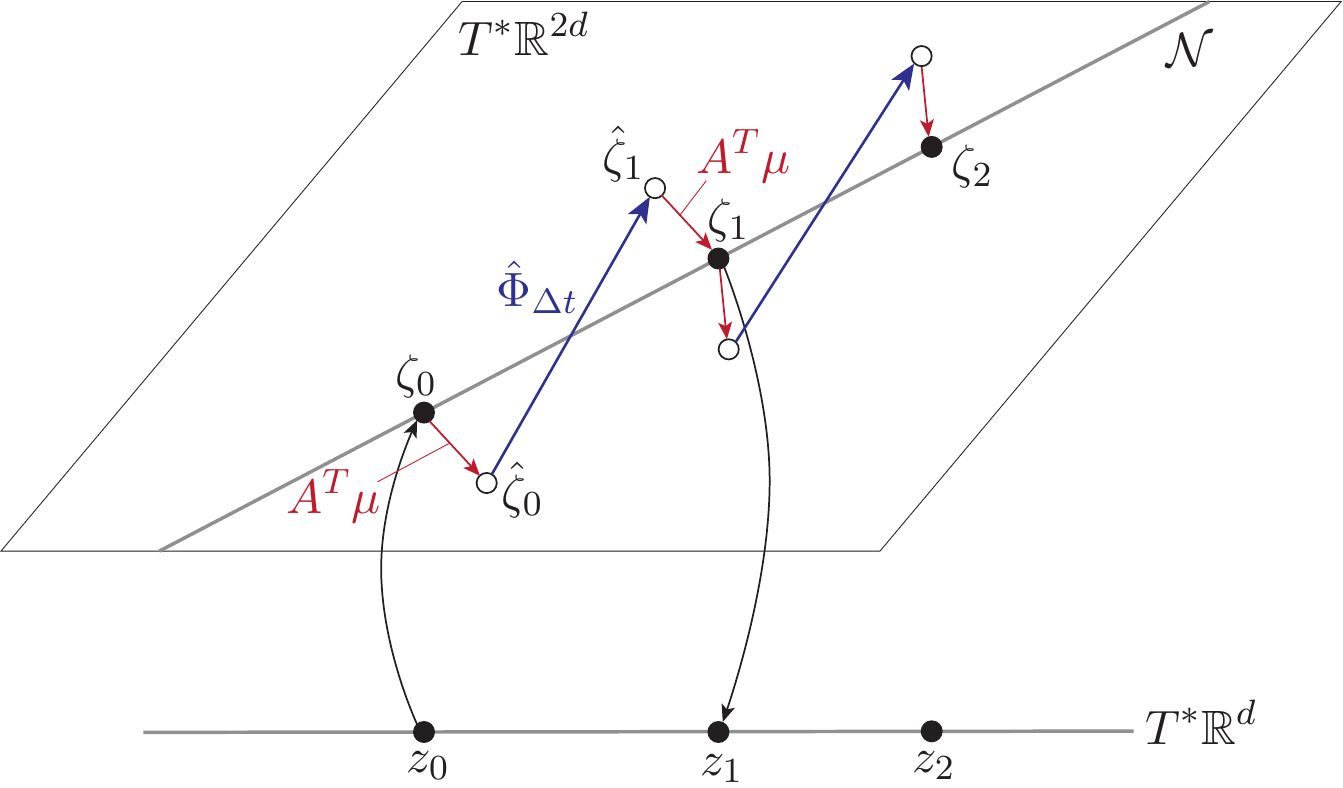}
  \caption{Extended phase space integrator with symmetric projection.}
  \label{fig:scheme}
\end{figure}

\subsection{Existence of Integrator}
\label{ssec:existence}
Let us show that the above method gives rise to an integrator $\Phi_{\dt}\colon T^{*}\R^{d} \to T^{*}\R^{d}$ for small enough $\dt > 0$.
To that end, let us first define
\begin{equation*}
  \iota\colon T^{*}\R^{d} \to \mathcal{N} \subset T^{*}\R^{2d};
  \qquad
  (q,p) \mapsto (q, q, p, p),
\end{equation*}
and also, for an arbitrary $\lambda \in \R^{2d}$,
\begin{equation*}
  \rho_{\lambda}\colon T^{*}\R^{2d} \to T^{*}\R^{2d};
  \qquad
  \zeta \mapsto \zeta + A^{T}\lambda.
\end{equation*}

In order to show that the discrete flow $\Phi_{\dt}\colon z_{n} \mapsto z_{n+1}$ exists, we need to show that, for a given $\zeta_{n} \in \mathcal{N}$, there exist $\zeta_{n+1} \in \mathcal{N}$ and $\mu \in \R^{2d}$ (a particular choice of $\lambda$ from above) satisfying \eqref{eq:nonlinear_eqs}, i.e.,
\begin{equation*}
  \iota \circ \Phi_{\dt}(z_{n}) =  \rho_{\mu} \circ \hat{\Phi}_{\dt} \circ \rho_{\mu} \circ \iota(z_{n}),
\end{equation*}
or equivalently, the diagram below commutes, where $\Psi_{\dt}$ will be defined in the proof of \Cref{prop:existence} to follow.
\begin{equation*}
  \begin{tikzcd}[column sep=8ex, row sep=7ex, ampersand replacement=\&]
    \hat{\zeta}_{n} = (\hat{q}_{n}, \hat{x}_{n}, \hat{p}_{n}, \hat{y}_{n}) \arrow[mapsto]{r}{\hat{\Phi}_{\dt}} \& \hat{\zeta}_{n+1} = (\hat{q}_{n+1}, \hat{x}_{n+1}, \hat{p}_{n+1}, \hat{y}_{n+1}) \arrow[mapsto,dashed]{d}{\rho_{\mu}} \\
    \zeta_{n} = (q_{n}, q_{n}, p_{n}, p_{n}) \arrow[mapsto,dashed]{u}{\rho_{\mu}} \& \zeta_{n+1} = (q_{n+1}, q_{n+1}, p_{n+1}, p_{n+1}) \\
    z_{n} = (q_{n}, p_{n}) \arrow[mapsto,dashed]{ru}{\Psi_{\dt}} \arrow[mapsto,dashed,swap]{r}{\Phi_{\dt}} \arrow[mapsto]{u}{\iota} \& z_{n+1} = (q_{n+1}, p_{n+1}) \arrow[mapsto]{u}{\iota}
  \end{tikzcd}
\end{equation*}

\begin{proposition}[Existence of discrete flow $\Phi$]
  \label{prop:existence}
  Let $\hat{\Phi}_{(\,\cdot\,)}\colon \R \times T^{*}\R^{2d} \to T^{*}\R^{2d}$ be the discrete flow defined by an integrator for the extended Hamiltonian system~\eqref{eq:Ham-extended}.
  For any $\bar{z}_{n} \in T^{*}\R^{d}$, there exist
  \begin{itemize}
  \item a neighborhood $U$ of $\bar{z}_{n}$ in $T^{*}\R^{d}$;\smallskip
  \item $\eps > 0$;\smallskip
  \item $\mu_{(\,\cdot\,)}\colon (-\eps,\eps) \times U \to \R^{2d}$ and $\Phi_{(\,\cdot\,)}\colon (-\eps,\eps) \times U \to T^{*}\R^{d}$
  \end{itemize}
  such that, for any $\dt \in (-\eps,\eps)$ and any $z_{n} \in U$,
  \begin{equation*}
    \iota \circ \Phi_{\dt}(z_{n}) = \hat{\Phi}_{\dt}\parentheses{ \iota(z_{n}) + A^{T}\mu_{\dt}(z_{n}) } + A^{T} \mu_{\dt}(z_{n}).
  \end{equation*}
\end{proposition}
\begin{proof}
  Let us define
  \begin{gather*}
    F\colon (\R \times T^{*}\R^{d}) \times (T^{*}\R^{2d} \times \R^{2d}) \to T^{*}\R^{2d} \times \R^{2d}
    \\
    F( (h,z_{n}), (\zeta,\lambda) )
    \defeq
    \begin{bmatrix}
      \zeta - \hat{\Phi}_{h}\parentheses{ \iota(z_{n}) + A^{T}\lambda } - A^{T} \lambda \smallskip\\
      A \zeta
    \end{bmatrix}.
  \end{gather*}
  Then the Jacobian
  \begin{equation*}
    \pd{F}{(\zeta,\lambda)}( (0,\bar{z}_{n}), (\iota(\bar{z}_{n}),0))
    = \begin{bmatrix}
      I_{2d} & -2A^{T} \\
      A & 0
    \end{bmatrix}
  \end{equation*}
  is invertible because $A$---defined in \eqref{eq:A}---is full-rank.
  Therefore, by the Implicit Function Theorem, there exist a neighborhood $U \subset T^{*}\R^{d}$ of $\bar{z}_{n}$, a positive number $\eps > 0$, and
  \begin{equation*}
    \Psi_{(\,\cdot\,)} \colon (-\eps,\eps) \times U \to T^{*}\R^{2d},
    \qquad
    \mu_{(\,\cdot\,)} \colon (-\eps,\eps) \times U \to \R^{2d}
  \end{equation*}
  such that, for any $(\dt,z_{n}) \in (-\eps,\eps) \times U$,
  \begin{equation*}
    F\parentheses{ (\dt,z_{n}), (\Psi_{\dt}(z_{n}), \mu_{\dt}(z_{n})) } = 0,
  \end{equation*}
  that is,
  \begin{equation*}
    \Psi_{\dt}(z_{n}) = \hat{\Phi}_{\dt}\parentheses{ \iota(z_{n}) + A^{T}\mu_{\dt}(z_{n}) } + A^{T} \mu_{\dt}(z_{n}),
    \qquad
    A^{T} \Psi_{\dt}(z_{n}) = 0.
  \end{equation*}
  However, since the second equality says $\Psi_{\dt}(z_{n}) \in \mathcal{N}$, there exists
  \begin{equation*}
    \Phi_{\dt}\colon U \to T^{*}\R^{d}
    \quad\text{such that}\quad
    \Psi_{\dt}(z_{n}) = \iota \circ \Phi_{\dt}(z_{n}). \qedhere
  \end{equation*}
\end{proof}

\subsection{Symmetry and Order of $\Phi$}
\label{ssec:symmetry}
We can also show that our method $\Phi$ is symmetric assuming that the extended phase space integrator $\hat{\Phi}$ is symmetric:
\begin{proposition}[Symmetry and order of $\Phi$]
  \label{prop:symmetry_order}
  If the extended phase space integrator $\hat{\Phi}$ is symmetric, then the integrator $\Phi$ defined in \Cref{prop:existence} is symmetric, i.e., $\Phi_{\dt}^{-1} = \Phi_{-\dt}$ for any $\dt \in (-\eps,\eps)$ with $\eps > 0$ defined in \Cref{prop:existence}.
  Moreover, $\Phi$ has the same order of accuracy as $\hat{\Phi}$.
\end{proposition}
\begin{proof}
  Let us first prove the symmetry.
  It is proved in \citet[p.~163]{HaLuWa2006} that a symmetric method combined with the symmetric projection method yields a symmetric method.
  In our setting, this implies that $\phi_{\dt}\colon \mathcal{N} \to \mathcal{N}$ defined by $\zeta_{n} \mapsto \zeta_{n+1}$ (see the diagram below) is symmetric, i.e., $\phi_{-\dt} = \phi_{\dt}^{-1}$.
  \begin{equation*}
    \begin{tikzcd}[column sep=12ex, row sep=7ex, ampersand replacement=\&]
      \zeta_{n} = (q_{n}, q_{n}, p_{n}, p_{n}) \arrow[mapsto, shift left=.5ex]{r}{\phi_{\dt}} \& \zeta_{n+1} = (q_{n+1}, q_{n+1}, p_{n+1}, p_{n+1}) \arrow[mapsto, shift left=.5ex]{l}{\phi_{\dt}^{-1} = \phi_{-\dt}}\\
      z_{n} = (q_{n}, p_{n}) \arrow[mapsto, swap, shift right=0.5ex]{r}{\Phi_{\dt}} \arrow[leftrightarrow]{u}{} \& z_{n+1} = (q_{n+1}, p_{n+1}) \arrow[leftrightarrow]{u}{} \arrow[mapsto, swap, shift right=0.5ex]{l}{\Phi_{\dt}^{-1}}
    \end{tikzcd}
  \end{equation*}
  However, $\phi$ and $\Phi$ are essentially the same map because one may identify $\zeta_{n}$ and $\zeta_{n+1}$ with $z_{n}$ and $z_{n+1}$, respectively, as shown in the diagram.
  In other words, the above diagram commutes.
  Thus the symmetry of $\phi$ implies that of $\Phi$.
  
  For the order of accuracy, again a result from \cite[p.~163]{HaLuWa2006} applied to our setting implies that $\phi$ solves the extended system~\eqref{eq:Ham-extended} on $\mathcal{N}$ with the same order of accuracy as $\hat{\Phi}$.
  However, as we have discussed in \Cref{ssec:extend_phase_space}, the extended system \eqref{eq:Ham-extended} on $\mathcal{N}$ is two identical copies of the original system~\eqref{eq:Ham}.
  Therefore, this implies that $\Phi$ solves \eqref{eq:Ham} with the same order of accuracy as well.
\end{proof}

\subsection{Is $\Phi_{\dt}$ Symplectic?}
\label{ssec:symplecticity0}
Here we shall perform some calculations to show that
\begin{equation*}
  \d{q}_{n+1} \wedge \d{p}_{n+1} = \d{q}_{n} \wedge \d{p}_{n}. 
\end{equation*}
One may consider the calculations to follow as a proof of symplecticity.
However, the argument to follow hardly reveals the geometry behind it because it is somewhat intricate.
Therefore, in \Cref{ssec:symplecticity} below, we shall present the underlying geometry of the method, and give a more geometrically sound proof of the symplecticity than what follows.

Let us write $\mu = (\mu_{1}, \mu_{2}) \in \R^{2d}$ with $\mu_{1}, \mu_{2} \in \R^{d}$.
Then we have
\begin{equation}
  \label{eq:zeta_n+1}
  \zeta_{n+1}
  = (q_{n+1}, q_{n+1}, p_{n+1}, p_{n+1})
  = (
  \hat{q}_{n+1} + \mu_{1},\,
  \hat{x}_{n+1} - \mu_{1},\,
  \hat{p}_{n+1} + \mu_{2},\,
  \hat{y}_{n+1} - \mu_{2}
  ),
\end{equation}
and so
\begin{equation*}
  q_{n+1} - \hat{q}_{n+1} = \mu_{1} = \hat{x}_{n+1} - q_{n+1},
  \qquad
   p_{n+1} - \hat{p}_{n+1} = \mu_{2} = \hat{y}_{n+1} - p_{n+1}.
\end{equation*}
Therefore,
\begin{equation*}
  q_{n+1} = \frac{1}{2}(\hat{q}_{n+1} + \hat{x}_{n+1}),
  \qquad
  p_{n+1} = \frac{1}{2}(\hat{p}_{n+1} + \hat{y}_{n+1}).
\end{equation*}
and thus
\begin{align}
  \label{eq:Omega1}
  \d{q}_{n+1} \wedge \d{p}_{n+1}
  &= \frac{1}{4}\parentheses{
    \d\hat{q}_{n+1} \wedge \d\hat{p}_{n+1} + \d\hat{x}_{n+1} \wedge \d\hat{y}_{n+1}
    + \d\hat{q}_{n+1} \wedge \d\hat{y}_{n+1} + \d\hat{x}_{n+1} \wedge \d\hat{p}_{n+1}
  } \nonumber\\
  &= \frac{1}{4}\parentheses{
    \d\hat{q}_{n} \wedge \d\hat{p}_{n} + \d\hat{x}_{n} \wedge \d\hat{y}_{n}
    + \d\hat{q}_{n+1} \wedge \d\hat{y}_{n+1} + \d\hat{x}_{n+1} \wedge \d\hat{p}_{n+1}
  },
\end{align}
where we also used the symplecticity of $\hat{\Phi}_{\dt}\colon \hat{\zeta}_{n} \mapsto \hat{\zeta}_{n+1}$ on the extended phase space $T^{*}\R^{2d}$, i.e.,
\begin{equation}
  \label{eq:symplecticity-hatPhi}
  \d\hat{q}_{n+1} \wedge \d\hat{p}_{n+1} + \d\hat{x}_{n+1} \wedge \d\hat{y}_{n+1}
  = \d\hat{q}_{n} \wedge \d\hat{p}_{n} + \d\hat{x}_{n} \wedge \d\hat{y}_{n}.
\end{equation}

Now, from \eqref{eq:zeta_n+1}, we also have
\begin{align*}
  \hat{x}_{n+1} - \hat{q}_{n+1} &= (q_{n+1} + \mu_{1}) - (q_{n+1} - \mu_{1}) = 2\mu_{1}, \\
  \hat{y}_{n+1} - \hat{p}_{n+1} &= (p_{n+1} + \mu_{2}) - (p_{n+1} - \mu_{2}) = 2\mu_{2}.
\end{align*}
On the other hand, recall that we also have
\begin{equation}
  \label{eq:zeta_n}
  \hat{\zeta}_{n}
  = (\hat{q}_{n},  \hat{x}_{n}, \hat{p}_{n}, \hat{y}_{n})
  = (
  q_{n} + \mu_{1},\,
  q_{n} - \mu_{1},\,
  p_{n} + \mu_{2},\,
  p_{n} - \mu_{2}
  ),
\end{equation}
and so, similarly,
\begin{equation*}
  \hat{q}_{n} - \hat{x}_{n} = (q_{n} + \mu_{1}) - (q_{n} - \mu_{1}) = 2\mu_{1},
  \qquad
  \hat{p}_{n} - \hat{y}_{n} = (p_{n} + \mu_{2}) - (p_{n} - \mu_{2}) = 2\mu_{2}.
\end{equation*}
Therefore, we obtain
\begin{equation*}
  \hat{x}_{n+1} - \hat{q}_{n+1} = \hat{q}_{n} - \hat{x}_{n},
  \qquad
  \hat{y}_{n+1} - \hat{p}_{n+1} = \hat{p}_{n} - \hat{y}_{n},
\end{equation*}
and so
\begin{equation*}
  \d(\hat{x}_{n+1} - \hat{q}_{n+1}) \wedge \d(\hat{y}_{n+1} - \hat{p}_{n+1})
  = \d(\hat{q}_{n} - \hat{x}_{n}) \wedge \d(\hat{p}_{n} - \hat{y}_{n}).
\end{equation*}
Subtracting \eqref{eq:symplecticity-hatPhi} from this identity,
\begin{equation*}
  \d\hat{q}_{n+1} \wedge \d\hat{y}_{n+1}
  + \d\hat{x}_{n+1} \wedge \d\hat{p}_{n+1}
  = \d\hat{q}_{n} \wedge \d\hat{y}_{n}
  + \d\hat{x}_{n} \wedge \d\hat{p}_{n}.
\end{equation*}
Therefore, \eqref{eq:Omega1} now gives
\begin{equation*}
  \d{q}_{n+1} \wedge \d{p}_{n+1}
  = \frac{1}{4}\parentheses{
    \d\hat{q}_{n} \wedge \d\hat{p}_{n} + \d\hat{x}_{n} \wedge \d\hat{y}_{n}
    + \d\hat{q}_{n} \wedge \d\hat{y}_{n} + \d\hat{x}_{n} \wedge \d\hat{p}_{n}
  },
\end{equation*}

On the other hand, \eqref{eq:zeta_n} also gives
\begin{equation*}
  q_{n} = \frac{1}{2}(\hat{q}_{n} + \hat{x}_{n}),
  \qquad
  p_{n} = \frac{1}{2}(\hat{p}_{n} + \hat{y}_{n}),
\end{equation*}
and thus
\begin{equation*}
  \d{q}_{n} \wedge \d{p}_{n}
  = \frac{1}{4}\parentheses{
    \d\hat{q}_{n} \wedge \d\hat{p}_{n} + \d\hat{x}_{n} \wedge \d\hat{y}_{n}
    + \d\hat{q}_{n} \wedge \d\hat{y}_{n} + \d\hat{x}_{n} \wedge \d\hat{p}_{n}
  }.
\end{equation*}
As a result, we obtain
\begin{equation*}
  \d{q}_{n+1} \wedge \d{p}_{n+1} = \d{q}_{n} \wedge \d{p}_{n}. \qedhere
\end{equation*}

\section{Geometry of Integrator and Symplecticity}
\label{sec:symplecticity}
As mentioned at the beginning of \Cref{ssec:symplecticity0}, the calculations performed above hardly reveals the geometry involved in the method in proving the symplecticity of the method.
In this section, we first provide a geometric interpretation of the method.
This will help us construct a more natural proof of symplecticity.

\subsection{Geometry of Integrator $\Phi_{\dt}$}
Building on \Cref{prop:existence}, we first would like to write our integrator $\Phi_{\dt}$ as a composition of some maps.

Let us first consider the step
\begin{equation}
  \label{eq:hatzeta_n-zeta_n}
  \zeta_{n} \mapsto \hat{\zeta}_{n} = \zeta_{n} + A^{T}\mu.
\end{equation}
We can show that $\hat{\zeta}_{n}$ is in a certain submanifold of $T^{*}\R^{2d}$:
\begin{lemma}
  Define
  \begin{equation}
    \label{eq:mathcalM}
    \mathcal{M} \defeq \setdef{ \zeta \in T^{*}\R^{2d} }{ A \circ \hat{\Phi}_{\dt}(\zeta) = -A \zeta },
  \end{equation}
  where $\hat{\Phi}_{\dt}$ is written in terms of the coordinates $(q, x, p, y)$.
  Then:
  \begin{enumerate}
    \renewcommand{\theenumi}{\roman{enumi}}
  \item For a small enough $|\dt| > 0$, $\mathcal{M}$ is a $2d$-dimensional submanifold of $T^{*}\R^{2d} \cong \R^{4d}$.
  \item $\hat{\zeta}_{n} \in \mathcal{M}$.
  \end{enumerate}
\end{lemma}
\begin{proof}
  (i)~Define
  \begin{equation*}
    G\colon T^{*}\R^{2d} \cong \R^{4d} \to \R^{2d};
    \qquad
    G(\zeta) \defeq A \circ \hat{\Phi}_{\dt}(\zeta) + A \zeta.
  \end{equation*}
  Then $\mathcal{M} = G^{-1}(0)$.
  So it suffices to show that the following differential of $G$ is full-rank:
  \begin{equation*}
    DG(\zeta) = A \circ D\hat{\Phi}_{\dt}(\zeta) + A = A \bigl( D\hat{\Phi}_{\dt}(\zeta) + I_{4d} \bigr).
  \end{equation*}
  Since $D\hat{\Phi}_{0} = I_{4d}$, for a small enough $|\dt| > 0$, $D\hat{\Phi}_{\dt}(\zeta) + I_{4d}$ is invertible.
  However, since $A$ is full-rank, this implies that $DG(\zeta)$ is full-rank for any $\zeta \in T^{*}\R^{2d}$ as well.
  
  (ii)~In view of \Cref{def:method} and \eqref{eq:nonlinear_eqs}, $\hat{\zeta}_{n}$ satisfies
  \begin{equation*}
    \zeta_{n+1} = \hat{\Phi}_{\dt}\bigl( \hat{\zeta}_{n} \bigr) + A^{T}\mu
    \quad\text{and}\quad
    A \zeta_{n+1} = 0.
  \end{equation*}
  However, since \eqref{eq:hatzeta_n-zeta_n} gives $A^{T} \mu = \hat{\zeta}_{n} - \zeta_{n}$, we have
  \begin{equation*}
    \zeta_{n+1} = \hat{\Phi}_{\dt}\bigl( \hat{\zeta}_{n} \bigr) + \hat{\zeta}_{n} - \zeta_{n}
    \quad\text{and}\quad
    A \zeta_{n+1} = 0.
  \end{equation*}
  Multiplying $A$ from the left to the first equality and noting that $\zeta_{n}, \zeta_{n+1} \in \mathcal{N} = \ker A$, we obtain
  \begin{equation*}
    0 = A \circ \hat{\Phi}_{\dt}\bigl( \hat{\zeta}_{n} \bigr) + A \hat{\zeta}_{n},
  \end{equation*}
  that is, $\hat{\zeta}_{n} \in \mathcal{M}$.
\end{proof}
As a result, using the map $\mu_{(\,\cdot\,)}\colon (-\eps,\eps) \times U \to T^{*}\R^{2d}$ introduced in \Cref{prop:existence}, this step is described by the map
\begin{equation*}
  \sigma_{\dt} \colon \mathcal{N} \to \mathcal{M};
  \qquad
  (q,q,p,p) \mapsto (q,q,p,p) + A^{T} \mu_{\dt}(q,p).
\end{equation*}
We shall also need the following embedding later:
\begin{equation*}
  i_{\mathcal{M}}\colon \mathcal{M} \hookrightarrow T^{*}\R^{2d}.
\end{equation*}

Let us next look at the step
\begin{equation*}
  \hat{\zeta}_{n+1} \mapsto \zeta_{n+1} = \hat{\zeta}_{n+1} + A^{T}\mu,
\end{equation*}
or
\begin{equation*}
  (\hat{q}_{n+1}, \hat{x}_{n+1}, \hat{p}_{n+1}, \hat{y}_{n+1})
  \mapsto
  (\hat{q}_{n+1} + \mu_{1},\,
  \hat{x}_{n+1} - \mu_{1},\,
  \hat{p}_{n+1} + \mu_{2},\,
  \hat{y}_{n+1} - \mu_{2}
  ).
\end{equation*}
Notice that, since $\zeta_{n+1} \in \mathcal{N}$, we may equate the first two terms as well as the last two on the right-hand side to have
\begin{equation*}
  \mu_{1} = \frac{1}{2}(\hat{x}_{n+1} - \hat{q}_{n+1})
  \qquad
  \mu_{2} = \frac{1}{2}(\hat{y}_{n+1} - \hat{p}_{n+1}).
\end{equation*}
Therefore, one may rewrite this step as
\begin{equation*}
  (\hat{q}_{n+1}, \hat{x}_{n+1}, \hat{p}_{n+1}, \hat{y}_{n+1})
  \mapsto
  \parentheses{
    \frac{\hat{q}_{n+1} + \hat{x}_{n+1}}{2},
    \frac{\hat{q}_{n+1} + \hat{x}_{n+1}}{2},
    \frac{\hat{p}_{n+1} + \hat{y}_{n+1}}{2},
    \frac{\hat{p}_{n+1} + \hat{y}_{n+1}}{2}
  }.
\end{equation*}
This motivates us to define the map
\begin{equation*}
  \kappa\colon T^{*}\R^{2d} \to \mathcal{N} \subset T^{*}\R^{2d};
  \qquad
  (q, x, p, y) \mapsto \parentheses{ \frac{q + x}{2}, \frac{q + x}{2}, \frac{p + y}{2}, \frac{p + y}{2} }.
\end{equation*}
so that
\begin{equation*}
  \zeta_{n+1} = \kappa( \hat{\zeta}_{n} ).
\end{equation*}

Finally, in order to describe the last step
\begin{equation*}
  \zeta_{n+1} = (q_{n}, q_{n}, p_{n}, p_{n})
  \mapsto
  z_{n} = (q_{n}, p_{n}),
\end{equation*}
let us define
\begin{equation*}
  \varpi\colon \mathcal{N} \to T^{*}\R^{d};
  \qquad
  (q, q, p, p) \mapsto (q, p).
\end{equation*}

Putting all of the maps together, we have the diagram below for our integrator $\Phi$.
\begin{equation*}
  \begin{tikzcd}[column sep=8ex, row sep=7ex, ampersand replacement=\&]
    T^{*}\R^{2d} \arrow{r}{\hat{\Phi}_{\dt}} \& T^{*}\R^{2d} \arrow{dd}{\kappa} \\
    \mathcal{M} \arrow[hook]{u}{i_{\mathcal{M}}} \&  \\
    \iota(U) \arrow{u}{\sigma_{\dt}} \subset \mathcal{N} \& \mathcal{N} \arrow{d}{\varpi} \\
    U \arrow[swap]{r}{\Phi_{\dt}} \arrow{u}{\iota} \& T^{*}\R^{d}
  \end{tikzcd}
  \qquad
  \begin{tikzcd}[column sep=8ex, row sep=7ex, ampersand replacement=\&]
    (\hat{q}_{n}, \hat{x}_{n}, \hat{p}_{n}, \hat{y}_{n}) \arrow[mapsto]{r}{} \& (\hat{q}_{n+1}, \hat{x}_{n+1}, \hat{p}_{n+1}, \hat{y}_{n+1}) \arrow[mapsto]{dd}{} \\
    (\hat{q}_{n}, \hat{x}_{n}, \hat{p}_{n}, \hat{y}_{n}) \arrow[mapsto]{u}{} \& \\
    (q_{n}, q_{n}, p_{n}, p_{n}) \arrow[mapsto]{u}{} \& (q_{n+1}, q_{n+1}, p_{n+1}, p_{n+1})  \arrow[mapsto]{d}{} \\
    (q_{n}, p_{n}) \arrow[mapsto,swap]{r}{} \arrow[mapsto]{u}{} \& (q_{n+1}, p_{n+1})
  \end{tikzcd}
\end{equation*}
In other words, we may write the integrator $\Phi_{\dt}$ as follows:
\begin{equation}
  \label{eq:Phi}
  \Phi_{\dt} = \varpi \circ \kappa \circ \hat{\Phi}_{\dt} \circ i_{\mathcal{M}} \circ \sigma_{\dt} \circ \iota.
\end{equation}

\subsection{Symplecticity of $\Phi_{\dt}$}
\label{ssec:symplecticity}
Now we are ready to prove our main result:
\begin{theorem}[Symplecticity of $\Phi_{\dt}$]
  \label{thm:symplecticity}
  The integrator $\Phi_{\dt}\colon T^{*}\R^{d} \to T^{*}\R^{d}$ is symplectic, i.e.,
  \begin{equation*}
    \Phi_{\dt}^{*} \Omega = \Omega.
  \end{equation*}
\end{theorem}

Let us first prove the following lemma:
\begin{lemma}
  \label{lem:Xi}  
  Let us define the following 2-form on $T^{*}\R^{2d}$:
  \begin{equation*}
    \Xi \defeq \d{q} \wedge \d{y} + \d{x} \wedge \d{p}.
  \end{equation*}
  Then
  \begin{equation}
    \label{eq:Omega-Xi}
    \kappa^{*} \varpi^{*} \Omega = \frac{1}{4}(\hat{\Omega} + \Xi)
  \end{equation}
  and
  \begin{equation}
    \label{eq:Xi-pullbacks}
    i_\mathcal{M}^{*} \hat{\Phi}_{\dt}^{*} \Xi = i_\mathcal{M}^{*} \Xi.
  \end{equation}
\end{lemma}
\begin{proof}
  Let us first prove \eqref{eq:Omega-Xi}.
  We may write
  \begin{equation*}
    \varpi \circ \kappa(q, x, p, y) = \parentheses{ \frac{q + x}{2}, \frac{p + y}{2} }.
  \end{equation*}
  So the pull-back of $\Omega$ by this map yields
  \begin{align*}
    \kappa^{*} \varpi^{*} \Omega
    &= (\varpi \circ \kappa)^{*} \Omega \\
    &= \frac{1}{4} (\d{q} + \d{x}) \wedge (\d{p} + \d{y}) \\
    &= \frac{1}{4} (\d{q} \wedge \d{p} + \d{x} \wedge \d{y} + \d{q} \wedge \d{y} + \d{x} \wedge \d{p}) \\
    &= \frac{1}{4}(\hat{\Omega} + \Xi).
  \end{align*}

  Let us next prove \eqref{eq:Xi-pullbacks}.
  It suffices to show that, for any $\zeta \in \mathcal{M}$ and any tangent vectors $v, w \in T_{\zeta}\mathcal{M}$,
  \begin{equation*}
    (\hat{\Phi}_{\dt}^{*} \Xi)_{\zeta}(v,w) = \Xi_{\zeta}(v,w).
  \end{equation*}
  
  Let us identify $v, w \in T_{\zeta}\mathcal{M}$ with vectors in $\R^{4d}$ written in terms of the coordinates
  \begin{equation*}
    \braces{
      \left. \pd{}{q}\right|_{\zeta},\,
      \left. \pd{}{x}\right|_{\zeta},\,
      \left. \pd{}{p}\right|_{\zeta},\,
      \left. \pd{}{y}\right|_{\zeta}
    }
  \end{equation*}
  for $T_{\zeta}(T^{*}\R^{2d}) \cong \R^{4d}$.
  In view of the definition~\eqref{eq:mathcalM} of the submanifold $\mathcal{M}$, one sees that $v, w \in T_{\zeta}\mathcal{M} \cong \R^{4d}$ satisfy
  \begin{equation*}
    A \circ D\hat{\Phi}_{\dt}(\zeta) v = -A v,
    \qquad
    A \circ D\hat{\Phi}_{\dt}(\zeta) w = -A w.
  \end{equation*}
  
  Now, we may write, using the definition of $\Xi$,
  \begin{equation*}
    \Xi_{\zeta}(v,w) = v^{T} \mathbb{X} w
    \quad
    \text{with}
    \quad
    \mathbb{X} \defeq
    \begin{bmatrix}
      0 & 0 & 0 & I_{d} \\
      0 & 0 & I_{d} & 0 \\
      0 & -I_{d} & 0 & 0 \\
      -I_{d} & 0 & 0 & 0
    \end{bmatrix}.
  \end{equation*}
  Let us set, for any $d \in \N$,
  \begin{equation*}
    \mathbb{J}_{2d} \defeq
    \begin{bmatrix}
      0 & I_{d} \\
      -I_{d} & 0
    \end{bmatrix}.
  \end{equation*}
  Then it is a straightforward calculation to see that
  \begin{equation*}
    A^{T} \mathbb{J}_{2d} A
    = \begin{bmatrix}
      0 & 0 & I_{d} & -I_{d} \\
      0 & 0 & -I_{d} & I_{d} \\
      -I_{d} & I_{d} & 0 & 0 \\
      I_{d} & -I_{d} & 0 & 0
    \end{bmatrix}
    = \mathbb{J}_{4d} - \mathbb{X}
    \iff
    \mathbb{X} = \mathbb{J}_{4d} - A^{T} \mathbb{J}_{2d} A.
  \end{equation*}

  Notice also that $\hat{\Phi}_{\dt} \colon T^{*}\R^{2d} \to T^{*}\R^{2d}$ is symplectic, i.e.,
  \begin{equation*}
    D\hat{\Phi}_{\dt}(\zeta)^{T} \mathbb{J}_{4d} D\Phi_{\dt}(\zeta) = \mathbb{J}_{4d}.
  \end{equation*}
  Therefore, we have
  \begin{align*}
    (\hat{\Phi}_{\dt}^{*} \Xi)_{\zeta}(v,w)
    &= (D\hat{\Phi}_{\dt}(\zeta) v)^{T} \mathbb{X} (D\hat{\Phi}_{\dt}(\zeta) w) \\
    &= v^{T} D\hat{\Phi}_{\dt}(\zeta)^{T} \mathbb{J}_{4d} D\hat{\Phi}_{\dt}(\zeta) w
      - (v^{T} D\hat{\Phi}_{\dt}(\zeta)^{T}) A^{T} \mathbb{J}_{2d} A (D\hat{\Phi}_{\dt}(\zeta) w) \\
    &= v^{T} \mathbb{J}_{4d} w
       - (A \circ D\hat{\Phi}_{\dt}(\zeta) v)^{T} \mathbb{J}_{2d} (A \circ D\hat{\Phi}_{\dt}(\zeta) w)\\
    &=  v^{T} \mathbb{J}_{4d} w
       - (-A v)^{T} \mathbb{J}_{2d} (-A w)\\
    &=  v^{T} \mathbb{J}_{4d} w
      - v^{T} A^{T} \mathbb{J}_{2d} A w\\
    &=  v^{T} \mathbb{J}_{4d} w
      - v^{T} (\mathbb{J}_{4d} - \mathbb{X}) w\\
    &= v^{T} \mathbb{X} w \\
    &= \Xi_{\zeta}(v,w). \qedhere
  \end{align*}
\end{proof}

Now we are ready to prove the main result.
\begin{proof}[Proof of \Cref{thm:symplecticity}]
  Since we may write
  \begin{equation*}
    i_{\mathcal{M}} \circ \sigma_{\dt} \circ \iota (q, p)
    = (q + \mu_{1},\, q - \mu_{1},\, p + \mu_{2},\, p - \mu_{2}),
  \end{equation*}
  we have
  \begin{align*}
    \iota^{*} \sigma_{\dt}^{*} i_{\mathcal{M}}^{*} (\hat{\Omega} + \Xi)
    &= (\d{q} + \d\mu_{1}) \wedge (\d{p} + \d\mu_{2})
      + (\d{q} - \d\mu_{1}) \wedge (\d{p} - \d\mu_{2}) \\
    &\quad+ (\d{q} + \d\mu_{1}) \wedge (\d{p} - \d\mu_{2})
      + (\d{q} - \d\mu_{1}) \wedge (\d{p} + \d\mu_{2}) \\
    &= 4 \d{q} \wedge \d{p} \\
    &= 4 \Omega.
  \end{align*}
  
  Therefore, using the definition~\eqref{eq:Phi}, the equalities \eqref{eq:Omega-Xi} and \eqref{eq:Xi-pullbacks} from \Cref{lem:Xi}, the symplecticity of $\hat{\Phi}_{\dt}$ with respect to $\hat{\Omega}$, and the above equality, we have
  \begin{align*}
    \Phi_{\dt}^{*} \Omega
    &= \iota^{*} \sigma_{\dt}^{*} i_{\mathcal{M}}^{*} \hat{\Phi}_{\dt}^{*} \kappa^{*} \varpi^{*} \Omega \\
    &= \frac{1}{4} \iota^{*} \sigma_{\dt}^{*} i_{\mathcal{M}}^{*} \hat{\Phi}_{\dt}^{*} (\hat{\Omega} + \Xi) \\
    &= \frac{1}{4} \iota^{*} \sigma_{\dt}^{*} i_{\mathcal{M}}^{*} (\hat{\Omega} + \Xi) \\
    &= \Omega. \qedhere
  \end{align*}
\end{proof}

\section{Implementation and Numerical Results}
\label{sec:results}
\subsection{Implementation}
\label{ssec:implementation}
Our base $2^{\rm nd}$-order method uses the extended phase space integrator~\eqref{eq:Strang} of \citet{Pi2015} using the Strang splitting.
We can construct higher-order extended phase space integrators from \eqref{eq:Strang} using the symmetric Triple Jump composition (see \cite{CrGo1989,Forest1989,Su1990,Yo1990}; also \citet[Example~II.4.2]{HaLuWa2006}).
For example, denoting the above $2^{\rm nd}$-order method in \eqref{eq:Strang} by $\hat{\Phi}^{(2)}_{\dt}$, we recursively construct an $n^{\rm th}$-order ($n$ being even) method as follows:
\begin{equation}
  \label{eq:nth-order-tripleJump}
  \hat{\Phi}^{(n)}_{\dt} \defeq \hat{\Phi}^{(n-2)}_{\gamma_{3}\dt} \circ \hat{\Phi}^{(n-2)}_{\gamma_{2}\dt} \circ \hat{\Phi}^{(n-2)}_{\gamma_{1}\dt},
\end{equation}
where
\begin{equation*}
  \gamma_{1} = \gamma_{3} \defeq \frac{1}{2 - 2^{1/(n-1)}},
  \qquad
  \gamma_{2} \defeq -\frac{2^{1/(n-1)}}{2 - 2^{1/(n-1)}}.
\end{equation*}
This results in a $3^{n/2}$-stage method of order $n$.
For example, the $4^{\rm th}$ order Triple Jump composition is a $9$-stage method, whereas the $6^{\rm th}$-order method has $27$ stages.

Another composition technique we have considered is the fractals method of \citet{Su1990} (see also \cite[Example~II.4.3]{HaLuWa2006}), i.e., instead of \eqref{eq:nth-order-tripleJump},
\begin{equation}
  \label{eq:nth-order-Suzuki}
  \hat{\Phi}^{(n)}_{\dt} \defeq \hat{\Phi}^{(n-2)}_{\gamma_{5}\dt} \circ \hat{\Phi}^{(n-2)}_{\gamma_{4}\dt} \circ  \hat{\Phi}^{(n-2)}_{\gamma_{3}\dt} \circ \hat{\Phi}^{(n-2)}_{\gamma_{2}\dt} \circ \hat{\Phi}^{(n-2)}_{\gamma_{1}\dt}
\end{equation}
with
\begin{equation*}
  \gamma_{1} = \gamma_{2} = \gamma_{4} = \gamma_{5} \defeq \frac{1}{4 - 4^{1/(n-1)}},
  \qquad
  \gamma_{3} \defeq -\frac{4^{1/(n-1)}}{4 - 4^{1/(n-1)}}.
\end{equation*}
We also used Yoshida's method~\cite{Yo1990} (see also \cite[Section~V.3.2]{HaLuWa2006}), which is a $6^{\rm th}$-order symmetric composition technique with 7 stages.

For the symmetric projection step, recall that we need to solve the nonlinear equations
\begin{equation*}
  F_{\dt}(\zeta_{n+1}, \mu) \defeq
  \begin{bmatrix}
    \zeta_{n+1} - \hat{\Phi}_{\dt}(\zeta_{n} + A^{T}\mu) - A^{T}\mu \\
    A \zeta_{n+1}
  \end{bmatrix}
  = 0
\end{equation*}
for $(\zeta_{n+1}, \mu)$.
However, one may eliminate $\zeta_{n+1}$ to have the following nonlinear equation for $\mu$:
\begin{equation}
  \label{eq:f-projection}
  f_{\dt}(\mu) \defeq A\parentheses{ \hat{\Phi}_{\dt}(\zeta_{n} + A^{T}\mu) + A^{T}\mu } = 0.
\end{equation}
Its Jacobian is
\begin{equation*}
  Df_{\dt}(\mu) = A\, D\hat{\Phi}_{\dt}(\zeta_{n} + A^{T}\mu)\, A^{T} + A A^{T},
\end{equation*}
and becomes very simple for $\dt = 0$: Noting that $A A^{T} = 2 I_{2d}$, we have
\begin{equation*}
  Df_{0}(\mu) = 2A A^{T} = 4 I_{2d}.
\end{equation*}
As suggested by \citet[Section~V.4.1]{HaLuWa2006}, we exploit this simple structure of the Jacobian to construct the simplified Newton approximation
\begin{equation}
  \label{eq:simplified_Newton}
  \mu^{(k+1)} = \mu^{(k)}- (Df_{0})^{-1} f_{\dt}\bigl( \mu^{(k)} \bigr)
  = \mu^{(k)}- \frac{1}{4} f_{\dt}\bigl( \mu^{(k)} \bigr),
\end{equation}
where we start with $\mu^{(0)} = 0$.
As we shall see in \Cref{ssec:efficiency}, these simplified Newton iterations tend to converge quickly with a reasonable tolerance $\epsilon$ imposed so it stops after $N$ iterations
\begin{equation}
\label{eq:proj-error}
\norm{ \mu^{(N+1)} - \mu^{(N)} } < \epsilon,
\end{equation}
and we set $\mu = \mu^{(N)}$.

\subsection{Broyden's Method}
\label{ssec:Broyden}
While the above simplified Newton method~\eqref{eq:simplified_Newton} is computationally efficient and works well in most of the examples we shall show below, it is not clear if one can prove that such a method indeed converges to an actual solution of the nonlinear equations.
In fact, as we shall see in \Cref{appx:10vortices} (\Cref{tab:10_vortices_2nd_0.01_2}), the method seems to occasionally fail to converge after a reasonable number of iterations ($N\leq100$ in our simulations) depending on the problem and the values of $\dt$ and $\epsilon$.

Therefore, one may instead use the following quasi-Newton method called Broyden's method~\cite{broyden1965class}:
\begin{subequations}
  \label{eq:Broyden}
  \begin{equation}
    \label{eq:Broyden-update}
    \mu^{(k+1)} = \mu^{(k)} - J_{k}^{-1} f_{\dt}\bigl( \mu^{(k)} \bigr),
  \end{equation}
  where $\{ J_{k} \}_{k\ge0}$ are approximations to $\{ Df_{\dt}(\mu^{(k)}) \}_{k\ge0}$ defined by the initial guess $J_{0} \defeq Df_{0}(\mu^{(0)}) = 4I_{2d}$ and the updates
  \begin{equation*}
    J_{k} = J_{k-1} + \frac{\Delta f^{(k)} - J_{k-1} \Delta \mu^{(k)}}{\norm{\Delta \mu^{(k)}}^2} \bigl(\Delta \mu^{(k)}\bigr)^T
    \quad
    \text{for}
    \quad
    k \ge 1
  \end{equation*}
  with $\Delta f^{(k)} \defeq f(\mu^{(k)}) - f(\mu^{(k-1)})$ and $\Delta\mu^{(k)} \defeq \mu^{(k)} - \mu^{(k-1)}$.
  This guarantees a $q$-quadratic convergence to an actual solution $\mu_{*}$ in our setting, assuming that $\mu^{(0)}$ and $J_{0} = Df_{0}(\mu^{(0)})$ are close enough to $\mu_{*}$ and $Df_{\dt}(\mu_{*})$, respectively; see, e.g., \citet{gay1979some} and \citet{mannel2021order} for details.
  This condition is satisfied if $\dt$ is taken small enough.
  Indeed, notice that setting $\dt = 0$ yields $\mu_{*} = 0$ and also that $\mu_{*}$ depends smoothly on $\dt$ (assuming that the Hamiltonian $H$ is smooth).
  Therefore, if $\dt$ is taken small enough, then $\mu_{*} \approx \mu^{(0)}$ and so $Df_{\dt}(\mu_{*}) \approx Df_{0}(\mu^{(0)})$ as well.
  
  As \citet{broyden1965class} suggested, it is computationally efficient to calculate the inverse $J_{k+1}^{-1}$ directly from $J_{k}^{-1}$ using the Sherman--Morrison formula as follows:
  \begin{equation}
    \label{eq:Broyden-approx}
    J_{k}^{-1} = J_{k-1}^{-1} + \frac{\Delta \mu^{(k)} - J_{k-1}^{-1} \Delta f^{(k)}}{\bigl(\Delta \mu^{(k)}\bigr)^T J_{k-1}^{-1} \Delta f^{(k)}} \bigl(\Delta \mu^{(k)}\bigr)^T J_{k-1}^{-1}.
  \end{equation}
\end{subequations}
The update~\eqref{eq:Broyden-update} along with \eqref{eq:Broyden-approx} is the well-known \textit{good Broyden's method}.

As we shall show in \Cref{appx}, Broyden's method offers an improvement in convergence and works well even for those cases where the simplified method~\eqref{eq:simplified_Newton} seems to fail to converge in a reasonable number of iterations.

\subsection{Order of Method}
As proved in \Cref{prop:symmetry_order}, our method $\Phi$ has the same order of accuracy as the extended integrator $\hat{\Phi}$ used.
We would like to numerically demonstrate this result using the following simple example:
Consider the Hamiltonian system~\eqref{eq:Ham} on $T^{*}\R$ with the Hamiltonian
\begin{equation}
  \label{eq:exact-H}
  H(q,p)=\frac{1}{2} \bigl(q^2+1\bigr) \bigl(p^2+1\bigr)
\end{equation}
with initial condition $(q(0), p(0)) = (-3, 0)$; this is exact solvable as shown in \citet{Ta2016b}.

We implemented our proposed semiexplicit method of orders 2, 4, and 6 using the Triple Jump~\eqref{eq:nth-order-tripleJump} as well as $4^{\rm th}$- and $6^{\rm th}$-order methods using Suzuki’s composition~\eqref{eq:nth-order-Suzuki} and Yoshida's $6^{\rm th}$-order composition mentioned above.

For comparison, we also used the $2^{\rm nd}$- and $4^{\rm th}$-order Gauss--Legendre methods~(see, e.g., \citet[Section~II.1.3]{HaLuWa2006} and \citet[Table~6.4 on p.~154]{LeRe2004}); its $2^{\rm nd}$-order method is commonly known as the Implicit Midpoint method, whereas we refer to the $4^{\rm th}$-order method as the IRK4 for short here.
We note that these methods are known to be symplectic; see, e.g., \cite[Theorem VI.4.2]{HaLuWa2006} and \cite[Section~6.3.1]{LeRe2004}.

\Cref{fig:exact-orders} shows how the maximum relative error in the Hamiltonian depends on the time step for our method along with Tao's method~\cite{Ta2016b} of the same orders, as well as the Gauss--Legendre methods.
We observe that our semiexplicit method exhibits the desired orders of accuracy.
Note that the errors of the $6^{\rm th}$-order methods for smaller time steps seem to be affected by the errors in the projection step, where the tolerance is $\epsilon=10^{-15}$ here.
Notice also that our method is consistently more accurate than Tao's of the same order.

\begin{figure}[htbp]
  \centering
  \includegraphics[width=0.7\linewidth]{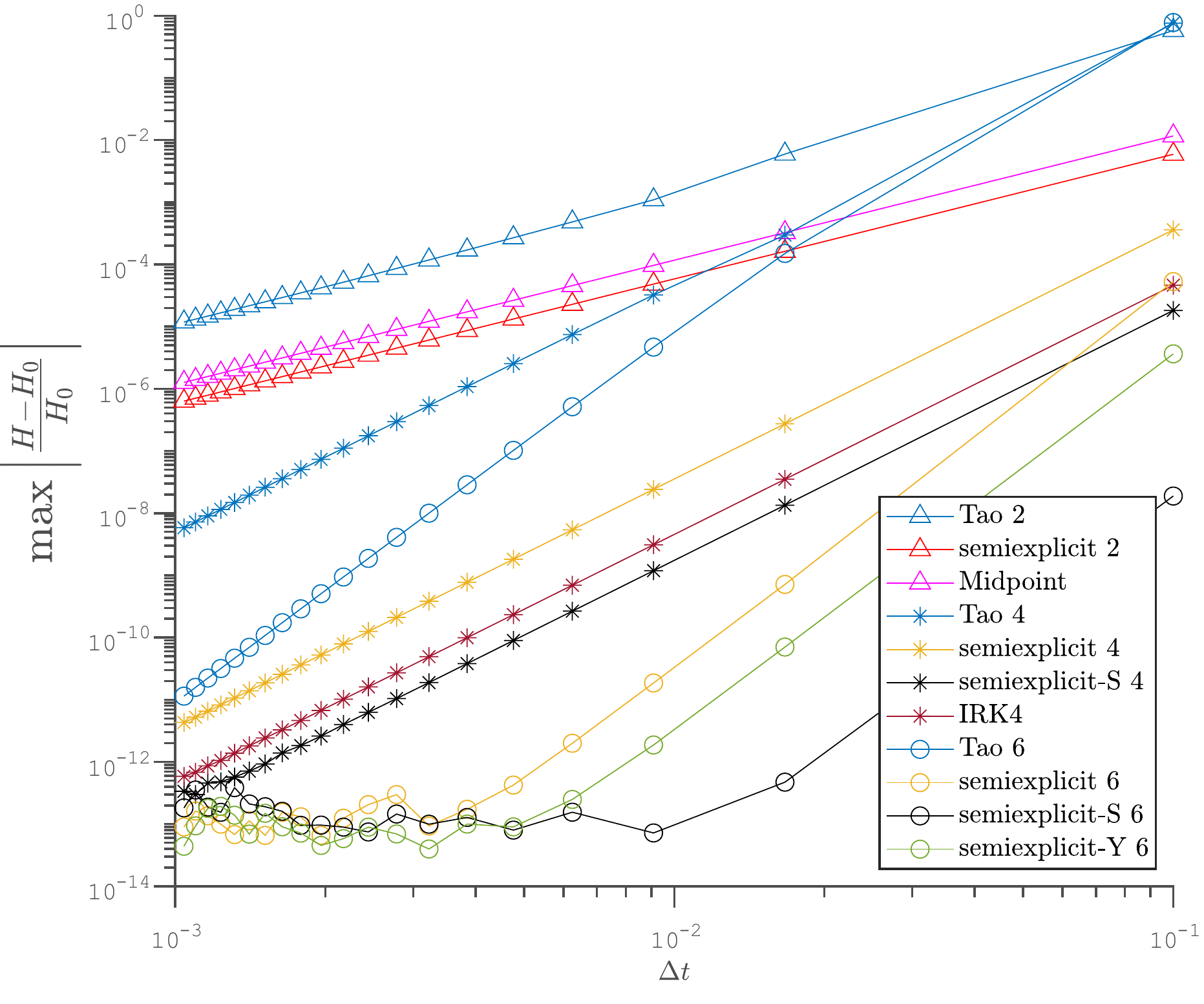}
  \caption{
    Maximum relative error in Hamiltonian~\eqref{eq:exact-H} (compared to initial value $H_{0}$) for different time steps with various methods:
    Tao~\cite{Ta2016b} (see \eqref{eq:Tao-Strang}) with $\omega=20$ using Triple Jump~\eqref{eq:nth-order-tripleJump}; proposed semiexplicit method with Triple Jump, Suzuki's and Yoshida's compositions; $2^{\rm nd}$-order Implicit Midpoint and the $4^{\rm th}$-order Gauss--Legendre method (IRK4).
    Simulations with terminal time $T=100$ and with a tolerance of $\epsilon=10^{-15}$ for the projection step were taken to calculate the maximum absolute relative error of Hamiltonian $H$ for time steps $\dt \in [0.001, 0.1]$.
  }
  \label{fig:exact-orders}
\end{figure}

\subsection{Preservation of Invariants}
\label{ssec:invariants}
Next we would like to address the long-time preservation of invariants of the system.
Before getting into the comparison of our method with Tao's \cite{Ta2016b}, let us first mention the main factor that affects the preservation of invariants.
As mentioned in \Cref{ssec:issues}, the main issue of the original extended phase space integrator of \citet{Pi2015} is the growth of the defect $(x - q, y - p)$.
Tao's modification seems to suppress this growth.
However, as we shall see below, the defect seems to affect the error of the the invariants.

Let us numerically demonstrate the behavior of the defect with an example.
Consider the following finite-dimensional approximation to the nonlinear Schr\"{o}dinger equation (NLS) equation (see \citet{CoKeStTaTa2010}) used in \citet{Ta2016b}.
It is a Hamiltonian system~\eqref{eq:Ham} with $d = N$ and the non-separable Hamiltonian
\begin{equation}
  \label{eq:Hamil-NLS}
  H(q,p) = \frac{1}{4} \sum_{i=1}^{N} \bigl(q_i^2 + p_i^2\bigr)^2 - \sum_{i=2}^{N} \bigl( p_{i-1}^2 p_{i}^2 + q_{i-1}^2 q_{i}^2 - q_{i-1}^2 p_{i}^2 - p_{i-1}^2 q_{i}^2 + 4p_{i-1}p_{i}q_{i-1}q_{i} \bigr).
\end{equation}
We note that this system has another invariant
\begin{equation}
  \label{eq:total_mass}
  I(q,p) \defeq \sum_{i=1}^{N} \bigl(q_i^2+p_i^2\bigr)
\end{equation}
called the total mass.
Following \citet{Ta2016b}, we set $N = 5$, $\omega = 100$, and $q(0) = (3,0.01,0.01,0.01,0.01)$ and $p(0) = (1,0,0,0,0)$.
\Cref{fig:Schrodinger_diff} shows how the Euclidean norm $\norm{ (q,p) - (x,y) }$ of the defect changes with time.
Note that higher-order methods of Tao~\eqref{eq:Tao-Strang} are constructed by the Tripe Jump composition~\eqref{eq:nth-order-tripleJump} as well.
One can see that, compared to Pihajoki's integrator~\eqref{eq:Strang}, Tao's middle step is able to successfully suppress the defect into small oscillations.

\begin{figure}[htbp]
  \centering
  \includegraphics[width=0.55\linewidth]{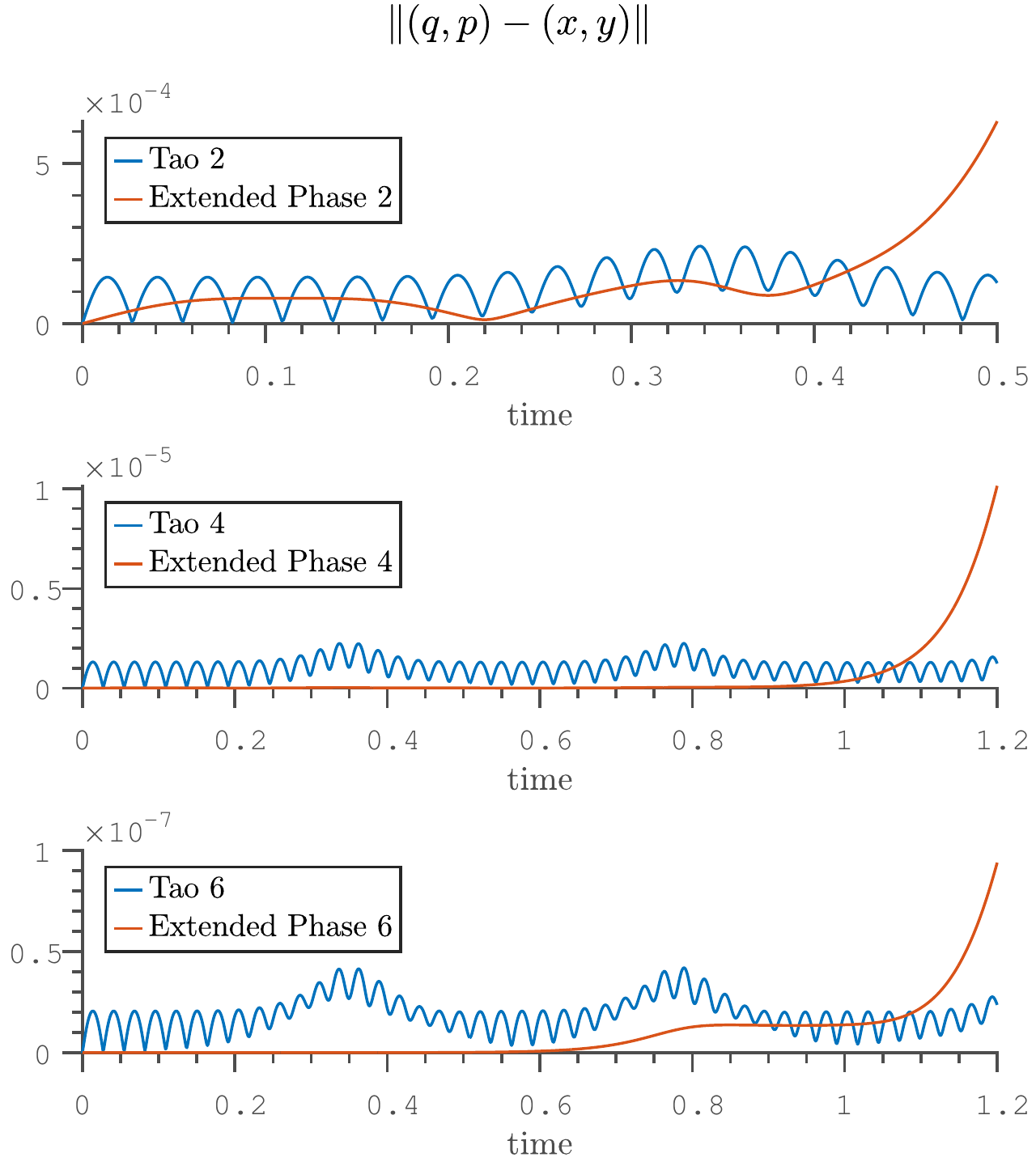}
  \caption{
    Defect $\norm{ (q,p) - (x,y) }$ in phase space copies with Pihajoki's original extended phase space integrator~\eqref{eq:Strang} and Tao's modified method~\eqref{eq:Tao-Strang} when solving NLS with $N=5$.
    Note that the defect for our semiexplicit method (not plotted) would be the same order as the tolerance $\epsilon$ ($10^{-10}$ or $10^{-13}$) used in the symmetric projection step, and are negligibly small compared to theirs.
  }
  \label{fig:Schrodinger_diff}
\end{figure}

However, this defect $(x-q,y-p)$, albeit small, seems to affect the accuracy of preserving invariants.
In order to demonstrate it, we solved the above NLS system using our semiexplicit method to compare the results with those of Tao's method.
We chose the time step as $\dt = 10^{-3}$ because this system exhibits weak turbulence (see~\citet{DyNePuZa1992}), and also used a relaxed tolerance of $\epsilon = 10^{-10}$ for the simplified Newton iterations~\eqref{eq:simplified_Newton} for faster computations.

\Cref{fig:Schrodinger_Inv_Tao} shows the time evolution of two invariants of the system---Hamiltonian and total mass $I$---for the semiexplicit method and Tao's method both with the Triple Jump composition technique.
While both methods preserve the invariants well without drifts, one clearly sees that our method has a much better accuracy.
In fact, the results seem to indicate that Tao's method picks up errors in the invariants roughly proportional to the defect $(x - q, y - p)$; compare with \Cref{fig:Schrodinger_diff}.

\begin{figure}[htbp]
  \centering
  \includegraphics[width=.8\linewidth]{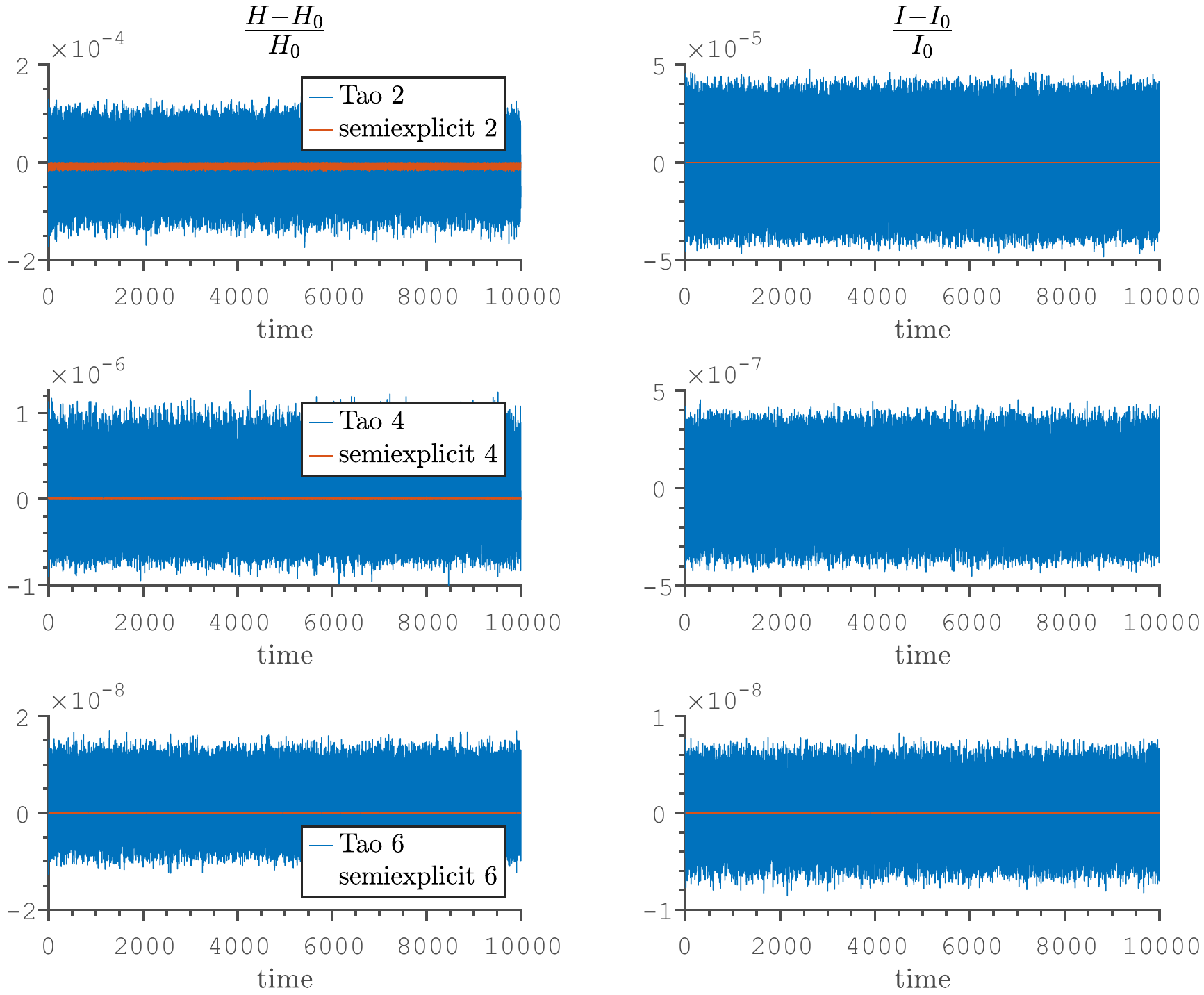}
  \caption{
    Comparison of preservation of two invariants ($H$ and $I$; see \eqref{eq:Hamil-NLS} and \eqref{eq:total_mass}) between Tao's method ($\omega=100$) and proposed semiexplicit method ($\epsilon = 10^{-10}$) for NLS system with $N=5$ and time step $\dt = 10^{-3}$.
    higher-order methods for both methods are constructed using Triple Jump technique \eqref{eq:nth-order-tripleJump}.
    The subscript $(\,\cdot\,)_{0}$ signifies the initial value.
  }
  \label{fig:Schrodinger_Inv_Tao}
\end{figure}

For a further demonstration, \Cref{fig:Schrodinger_Inv} shows a comparison of our method with the $2^{\rm nd}$- and $4^{\rm th}$-order Gauss--Legendre methods, for the same NLS problem as above but using a smaller tolerance of $\epsilon = 10^{-13}$ for both the semiexplicit and implicit methods.
We also used Suzuki's and Yoshida's methods for higher-order methods in addition to the Triple Jump method.

\begin{figure}[htbp]
  \centering
  \includegraphics[width=.8\linewidth]{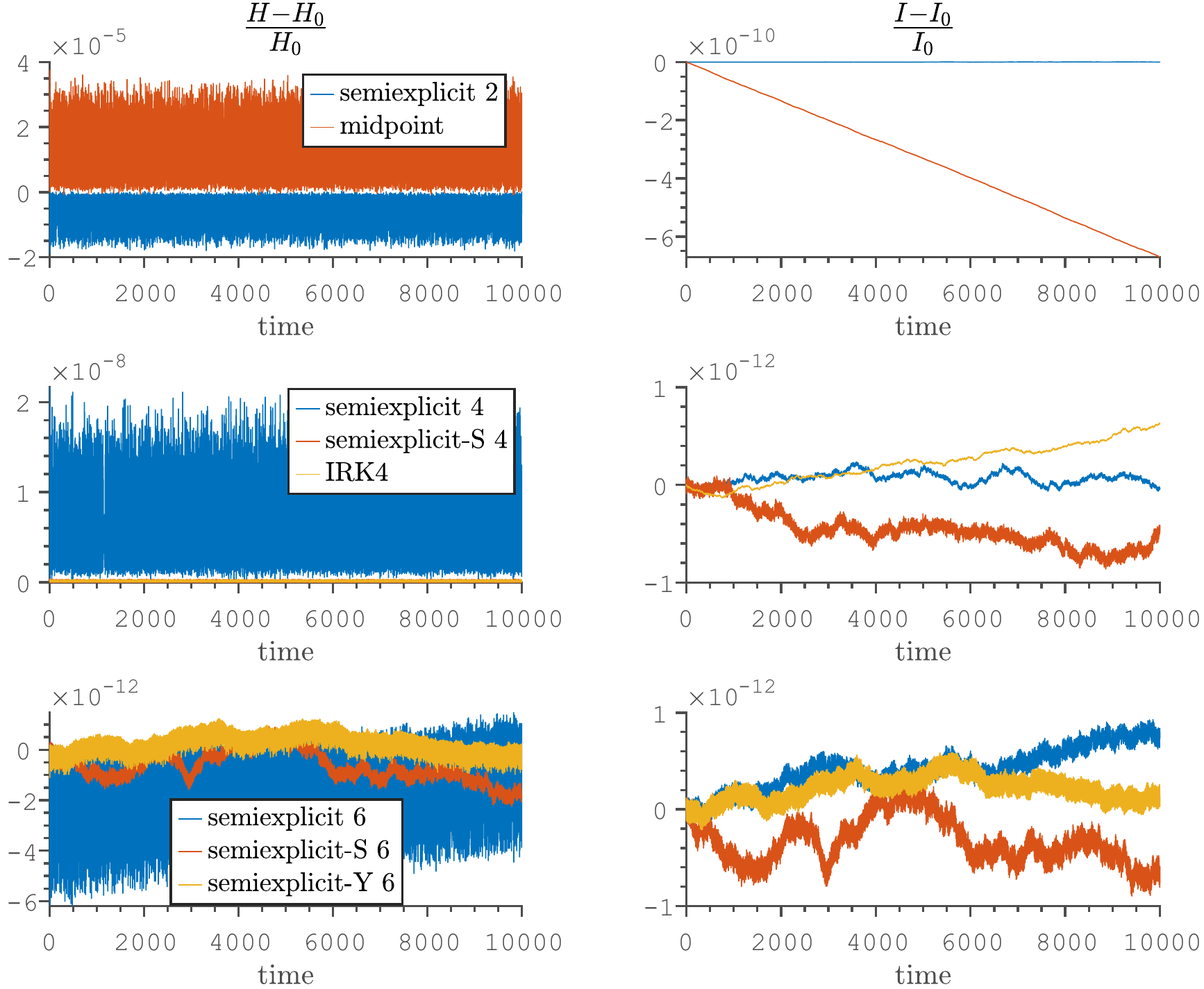}
  \caption{
    Comparison of same invariants from \Cref{fig:Schrodinger_Inv_Tao} with more strict tolerance of $\epsilon = 10^{-13}$, time step $\dt = 10^{-3}$ and different methods:
    Higher-order semiexplicit methods with Triple Jump and \citeauthor{Su1990}'s and \citeauthor{Yo1990}'s compositions (with labels ``S'' and ``Y'' respectively), as well as the $2^{\rm nd}$- and $4^{\rm th}$-order Gauss--Legendre methods (referred to as midpoint and IRK).
  }
  \label{fig:Schrodinger_Inv}
\end{figure}

Notice that, for the Implicit Midpoint (or the $2^{\rm nd}$-order Gauss--Legendre) method, one observes a drift in the total mass.
Since the Gauss--Legendre methods preserve quadratic invariants exactly (see \citet{Co1987} and \citet[Theorems~IV.2.1 and IV.2.2]{HaLuWa2006}), the culprit for the drift would be the accumulated error from the nonlinear solver (Newton's method here); the tolerance here is $\epsilon = 10^{-13}$.
Notice that our method exhibits no such drift, despite using the same tolerance with a simplified Newton method.

We also observe that Suzuki's composition has a better accuracy than the Triple Jump, especially in preserving the Hamiltonian for the $4^{\rm th}$-order methods, and also that Yoshida's $6^{\rm th}$-order method performs as well as Suzuki's, despite having much fewer stages and hence much faster as we shall see below.

As the third demonstration using another example, consider the motion of $N$ point vortices in $\mathbb{R}^2$ with circulations $\{ \Gamma_{i} \in \R\backslash\{0\} \}_{i=1}^{N}$~(see, e.g., \citet[Section~2.1]{Ne2001} and \citet[Section~2.1]{ChMa1993}). Interactions of them make their centers $\{ z_{i} \defeq (x_{i}, y_{i}) \in \R^{2} \}_{i=1}^{N}$ move according to the equations
\begin{equation*}
  \dot{x}_j = - \frac{1}{2\pi} \sum_{\substack{1\le i \le N\\i\neq j}} \frac{\Gamma_i (y_j - y_i)}{\|z_i - z_j\|^2},
  \qquad
  \dot{y}_j = \frac{1}{2\pi} \sum_{\substack{1\le i \le N\\i\neq j}} \frac{\Gamma_i (x_j - x_i)}{\|z_i - z_j\|^2}.
\end{equation*}
Taking the interaction energy of the vortex system given by
\begin{equation}
  \label{eq:Hamil-Vortices}
  H(x, y) \defeq -\frac{1}{4\pi} \sum_{i \neq j} \Gamma_i \Gamma_j \log\|z_i - z_j\|
\end{equation}
as the Hamiltonian, this system can be represented as a canonical Hamiltonian system~\eqref{eq:Ham} with $d = N$ by a simple coordinate transformation:
\begin{equation*}
  q_i \defeq \sqrt{|\Gamma_i|}\, x_i,
  \qquad
  p_i \defeq \sqrt{|\Gamma_i|}\, \sgn(\Gamma_i)\, y_i,
\end{equation*}
where $\sgn(x)$ is $1$ if $x > 0$ and $-1$ otherwise.
Note that the Hamiltonian is again non-separable.
This system has three invariants in addition to the Hamiltonian:
\begin{equation}
  \label{eq:impulses}
  Q \defeq \sum_{i=1}^{N} \Gamma_i x_i,
  \qquad
  P \defeq \sum_{i=1}^{N} \Gamma_i y_i.
  \qquad
  I \defeq \sum_{i=1}^{N} \Gamma_i \|z_i\|^2,
\end{equation}
where $(Q,P)$ is called the linear impulse, and $I$ is called the angular impulse.

Our test case has 10 vortices ($N=10$) with the following initial condition and circulations:
\begin{equation}
  \label{eq:IC_10-votex}
  \begin{split}
    x(0)      &=(3,-10,6,9,0,7,-8,5,9,7),\\
    y(0)      &=(-5,-6,0,-2,0,10,2,9,0,-1),\\
    \Gamma    &=\frac{1}{10}(-5,3,6,7,-2,-8,-9,-3,7,-6).
  \end{split}
\end{equation}

\begin{figure}[htbp]
  \centering
  \includegraphics[width=.9\linewidth]{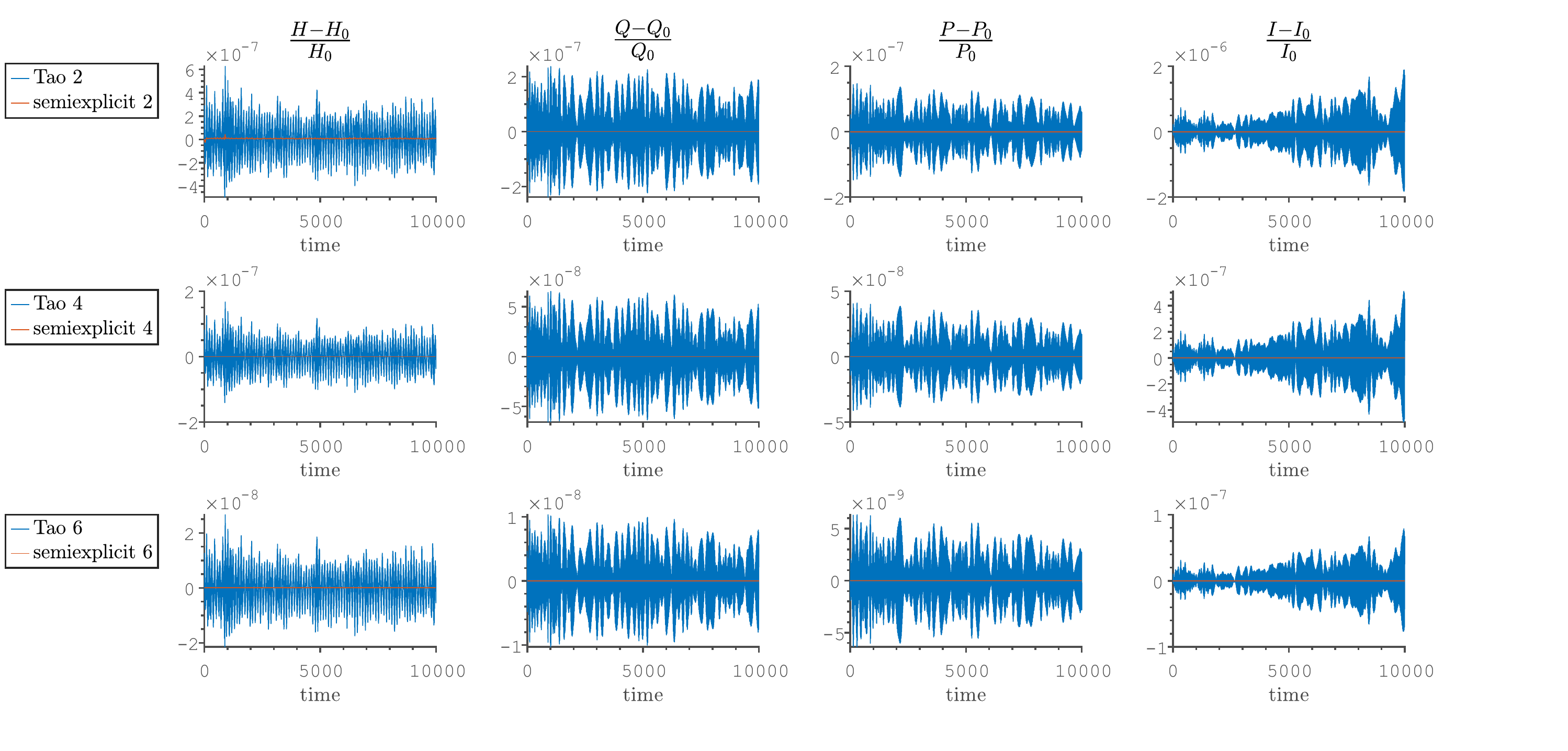}
  \caption{
    Comparison of preservation of four invariants ($H$, $Q$, $P$, and $I$; see \eqref{eq:Hamil-Vortices} and \eqref{eq:impulses}) between Tao's method ($\omega=7$) and proposed semiexplicit method ($\epsilon = 10^{-10}$) for the 10-vortex system with time step $\dt = 0.1$.
    higher-order methods for both methods are constructed using Triple Jump technique \eqref{eq:nth-order-tripleJump}.
    The subscript $(\,\cdot\,)_{0}$ signifies the initial value.
  }
  \label{fig:vortex_Inv_Tao}
\end{figure}

\Cref{fig:vortex_Inv_Tao} shows the results from solving this system using Tao's and our semiexplicit methods; the time step is $\dt = 0.1$ and the tolerance is $\epsilon = 10^{-10}$.
Again we observe that our semiexplicit method outperforms Tao's method in preserving the four invariants.

\begin{figure}[htbp]
  \centering
  \includegraphics[width=.9\linewidth]{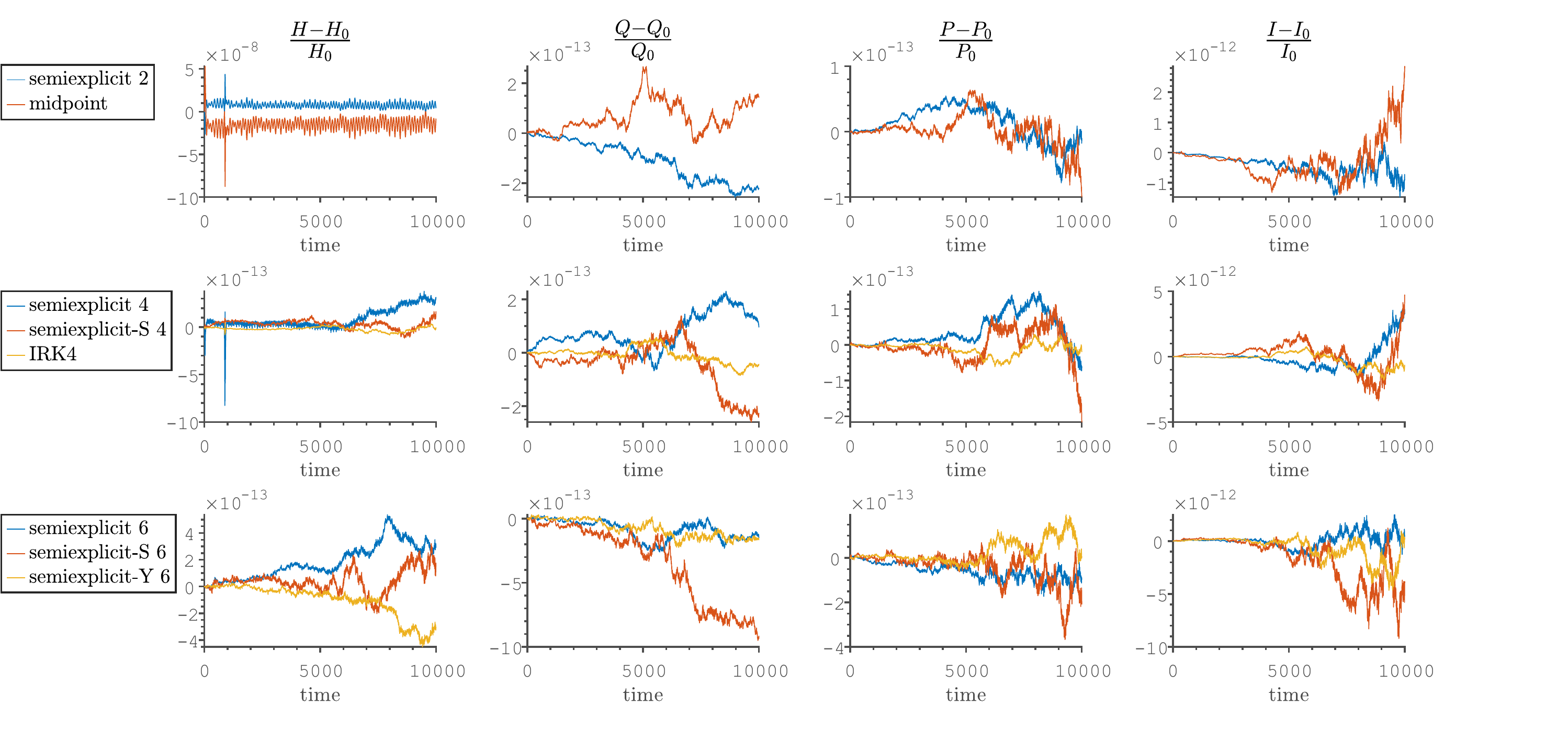}
  \caption{
    Comparison of same invariants from \Cref{fig:vortex_Inv_Tao} with more strict tolerance of $\epsilon = 10^{-13}$, time step $\dt = 0.1$ and various other methods:
    Higher-order semiexplicit methods with Triple Jump and \citeauthor{Su1990}'s and \citeauthor{Yo1990}'s compositions (with labels ``S'' and ``Y'' respectively), as well as the $2^{\rm nd}$- and $4^{\rm th}$-order Gauss--Legendre methods (referred to as midpoint and IRK).
  }
  \label{fig:vortex_Inv}
\end{figure}

\Cref{fig:vortex_Inv} shows further comparisons with other methods including the Gauss--Legendre methods with a more strict tolerance of $\epsilon = 10^{-13}$ for both the semiexplicit and implicit methods.
We see that the proposed semiexplicit method is comparable to the Gauss--Legendre methods in preserving all the four invariants.
Given that the Gauss--Legendre methods are known to preserve quadratic invariants exactly (see \citet{Co1987} and \citet[Theorems~2.1 and 2.2]{HaLuWa2006}), this result demonstrates that our method exhibits an excellent long-time behavior in preserving the invariants.

Also compare the errors of Tao's method in the invariants $Q$, $P$, and $I$ from \Cref{fig:vortex_Inv_Tao} and those of ours from \Cref{fig:vortex_Inv}.
The errors are significantly smaller with our method, indicating that the defect $(x - q, y - p)$ (and possibly also the symplecticity in the original phase space) affects the errors in the invariants.
This observation underscores the importance of our symmetric projection step.

\subsection{Computational Efficiency}
\label{ssec:efficiency}
Now that we have seen that our method exhibits desirable accuracies and long-time behaviors, a natural question to ask is whether it is as fast as Tao's explicit method, or at least faster than those implicit Gauss--Legendre methods.
So we would like to compare the computation times of those methods discussed above when applied to the NLS and the vortex problems using the same initial values and time steps from above.

\Cref{tab:Schrodinger} compares the computation times and average number of iterations of various methods applied to the NLS problem; see the caption for details.

\begin{table}[htbp]
  \caption{
    Comparison of computation times of various methods when solving NLS system with $N=5$ from \Cref{ssec:invariants} with time step $\dt = 10^{-3}$ and terminal time $T=1000$.
    Tao's method uses $\omega=100$; its higher-order versions (Tao~4 and Tao~6) use the Triple Jump technique, and so does our standard semiexplicit method (semiexplicit~4 and semiexplicit~6).
    Those referred to as semiexplicit-S and semiexplicit-Y are our semiexplicit methods using \citeauthor{Su1990}'s and \citeauthor{Yo1990}'s compositions, respectively.
    \texttt{Time} is the computation time in seconds averaged over $10$ simulations, whereas \texttt{NW\_itr} is the average number of Newton-type iterations used per step for both the simplified Newton~\eqref{eq:simplified_Newton} and Broyden's method~\eqref{eq:Broyden}.
    For the Midpoint and IRK4, the full Newton iteration is used.
    Computations were performed using MATLAB on a computer with Intel Core i7-8565U processor running at 1.80GHz.
  }
 \begin{tabular}{l|rc|rc||rc|rc}
   & \multicolumn{4}{c||}{$\epsilon=10^{-10}$} & \multicolumn{4}{c}{$\epsilon=10^{-13}$} \\ \cline{2-9}
   & \multicolumn{2}{c}{Simplified~\eqref{eq:simplified_Newton}} & \multicolumn{2}{|c||}{Broyden~\eqref{eq:Broyden}} & \multicolumn{2}{c}{Simplified~\eqref{eq:simplified_Newton}} & \multicolumn{2}{|c}{Broyden~\eqref{eq:Broyden}} \\ \hline
   \textbf{method} & \texttt{Time}  & \texttt{NW\_itr} & \texttt{Time} & \texttt{NW\_itr} & \texttt{Time} & \texttt{NW\_itr} & \texttt{Time} & \texttt{NW\_itr} \\ \hline
    \textbf{Tao 2}		& 7.03	 &	& 7.03	&     	 & 7.03	 & 	  & 7.03   &      \\
    \textbf{semiexplicit 2}	& 33.24	 & 3.37	& 42.12	& 3.40   & 44.71 & 5	  & 59.48  & 4.90 \\
    \textbf{Midpoint}		& 115.94 & 5.99	& 115.94 & 5.99  & 145.74 & 7.89  & 145.74 & 7.89 \\ \hline
    \textbf{Tao 4}		& 17.10	 &	& 17.10	&     	 & 17.10 &	  & 17.10  &   	  \\
    \textbf{semiexplicit 4}	& 32.50	 & 1.94	& 36.97	& 1.94   & 46.37 & 3	  & 56.23  & 3    \\
    \textbf{semiexplicit-S 4}	& 26.95	 & 1.09	& 26.96	& 1.09   & 50.13 & 2.28	  & 55.34  & 2.29 \\
    \textbf{IRK4}		& 235.30 & 5.21	& 235.30 & 5.21  & 299.09 &  6.98 & 299.09 & 6.98 \\ \hline
    \textbf{Tao 6}		& 44.75	 &	& 44.75	&     	 & 44.75 & 	  & 44.75  &      \\
    \textbf{semiexplicit 6}	& 37.90	 & 1.00	& 36.98	& 1.00   & 67.37 & 1.98	  & 72.34  & 1.98 \\
    \textbf{semiexplicit-S 6}	& 90.19	 & 1.00	& 88.83	& 1.00   & 86.30 & 1.00	  & 84.04  & 1.00 \\
    \textbf{semiexplicit-Y 6}	& 31.65	 & 1.00	& 30.68	& 1.00   & 31.46 & 1.02	  & 29.37  & 1.02 \\ \hline
  \end{tabular}

  \label{tab:Schrodinger}
\end{table}

The columns labeled by \texttt{NW\_itr} shows the average number of iterations required in the implicit part of each method.
We see that our semiexplicit method requires fewer iterations than the Gauss--Legendre methods (Midpoint and IRK4).
One may attribute the slowness of these implicit methods to full Newton's method used by them, but switching to quasi-Newton methods such as Broyden's method would not make these methods faster than ours because it is unlikely to curtail the disparity in the number of iterations.
We note that the number of iterations is a good measure of how these methods compare because the implicit methods and our methods solve nonlinear equations for the same number of unknowns: $(q,p) \in \R^{2d}$ for the implicit methods and $\mu \in \R^{2d}$ for ours.
Now, one may explain the disparity in the number of iterations as follows: Both use 0 as the initial guess, but the implicit methods need to solve for the $O(\dt)$ difference between $(q_{n},p_{n})$ and $(q_{n+1},p_{n+1})$, whereas ours needs to solve for $\mu$, which is essentially the defect after a single step; one may estimate it to be proportional to the local error of the explicit integrator employed, i.e., $O((\dt)^{n+1})$ for an $n$-th order method, because the defect is a quantity that vanishes for the exact solution.
This also explains why our method requires less iterations for higher-order methods.

Notice also that, although our $2^{\rm nd}$-order method is much slower than Tao's, it catches up with Tao's as the order of the method increases.
Particularly our $6^{\rm th}$-method with a relaxed tolerance of $\epsilon = 10^{-10}$ is faster than Tao's of the same order.
Moreover, recall from the last row of \Cref{fig:Schrodinger_Inv_Tao} (plots from the same problem) shows that our $6^{\rm th}$-order method is much more accurate than Tao's.
It is also notable that the semiexplicit method with Yoshida's composition is much faster than Tao's $6^{\rm th}$-order method with Triple Jump thanks to the fewer steps involved in Yoshida's composition.

\Cref{tab:10_vortices} shows the same comparison for the 10-vortex system.
We show results only with the simplified Newton method~\eqref{eq:simplified_Newton} here for the reason stated in the caption.
Note that the 10-vortex system is 20-dimensional whereas the NLS with $N=5$ from above is 10-dimensional.
We observe that our $4^{\rm th}$- and $6^{\rm th}$-order methods are faster than Tao's of the same orders even with $\epsilon = 10^{-13}$.

\begin{table}[hbtp]
  \caption{
    Comparison of computation times of various methods when solving the 10-vortex system from \Cref{ssec:invariants} with time step $\dt = 0.1$ and terminal time $T=1000$.
    We show results only with the simplified Newton method~\eqref{eq:simplified_Newton} because Broyden's method~\eqref{eq:Broyden} gives more or less the same result.
This is because the average number of iterations \texttt{NW\_itr} is usually very close to 1 here: these two methods share exactly the same first iteration step and so there is almost no difference between them when $\texttt{NW\_itr} \approx 1$.
    The other details are the same as \Cref{tab:Schrodinger} except $\omega = 7$ for Tao's method.
  }
  \begin{tabular}{l|rc||rc}
    & \multicolumn{2}{c||}{$\epsilon=10^{-10}$} & \multicolumn{2}{c}{$\epsilon=10^{-13}$} \\ \hline
    \textbf{method}         & \multicolumn{1}{c}{\texttt{Time}}        & \texttt{NW\_itr}       & \multicolumn{1}{c}{\texttt{Time}}        & \texttt{NW\_itr}       \\ \hline
    \textbf{Tao 2}			& 1.82	& 		& 1.82	&  			\\
    \textbf{semiexplicit 2}		& 3.00	& 2		& 4.29	& 3 			\\
    \textbf{Midpoint}		& 18.17	& 4.18	& 22.77	& 5.42		\\ \hline
    \textbf{Tao 4}			& 5.32	& 		& 5.32	&   			\\
    \textbf{semiexplicit 4}		& 4.10	& 1.00	& 4.25	& 1.05   		\\
    \textbf{semiexplicit-S 4}	& 6.59	& 1		& 6.76	& 1.01    		\\
    \textbf{IRK4}			& 45.99	& 4.01	& 57.05	& 4.99		\\ \hline
    \textbf{Tao 6}			& 15.63	& 		& 15.63	& 			\\
    \textbf{semiexplicit 6}		& 11.92	& 1		& 12.03	& 1.00 		\\
    \textbf{semiexplicit-S 6}	& 32.29	& 1		& 32.48	& 1.00		\\
    \textbf{semiexplicit-Y 6}	& 9.23	& 1		& 9.49	& 1.00	    	\\ \hline
  \end{tabular}
  \label{tab:10_vortices}
\end{table}

A possible explanation of why our method can be faster than Tao's for higher-order methods and higher-dimensional problems is the following:
Recall that Tao's $2^{\rm nd}$-order method~\eqref{eq:Tao-Strang} is a composition of 5 flows, whereas our explicit part~\eqref{eq:Strang} is a composition of 3 flows.
Therefore, for the $4^{\rm th}$-order method using the Triple Jump~\eqref{eq:nth-order-tripleJump}, Tao's involves 15 stages whereas ours involves 9 stages, and for the $6^{\rm th}$-order method, Tao's 45 stages whereas ours 27 stages.
Now, notice in \Cref{tab:Schrodinger,tab:10_vortices} that the number of Newton-type iterations, \texttt{NW\_itr}, gets smaller for higher-order methods and higher-dimensional problems, eventually becoming close to 1 in \Cref{tab:10_vortices}.
This means that only one step of Newton-type iteration is performed in the projection step most of the time.
In other words, in the simplified Newton iteration \eqref{eq:simplified_Newton}, $\mu^{(1)}$ is oftentimes taken as $\mu$, but then since $\mu^{(0)} = 0$, we have, using \eqref{eq:f-projection} and \eqref{eq:simplified_Newton},
\begin{equation*}
  \mu^{(1)} = -\frac{1}{4} A\,\hat{\Phi}_{\dt}(\zeta_{n}).
\end{equation*}
\textit{Note that $\hat{\Phi}_{\dt}(\zeta_{n})$ is already computed in the explicit part of the method, and so computing $\mu^{(1)}$ does not involve extra stages of discrete time evolutions (which would be required to compute $\mu^{(k)}$ for $k \ge 2$).}
Therefore, if $\texttt{NW\_itr} \approx 1$, the projection step becomes very simple and hence may require fewer computations compared to the extra stages in Tao's method.

Moreover, as the dimension of the problem increases, each stage becomes computationally more expensive, and so the computational advantage due to the difference in the number of stages would be amplified.

See \Cref{appx} for a further numerical study of the projection step.
The results from above and \Cref{appx} suggest the following:
\begin{itemize}
\item The projection step requires very few iterations when the time step $\dt$ is small enough compared to the scale of the problem.
  \smallskip
\item Increasing $\dt$ may result in a significant increase in the number of iterations, and hence a slowdown of our method; see \Cref{tab:Schrodinger_0.01} in \Cref{appx}.
  \smallskip
\item The simplified Newton method~\eqref{eq:simplified_Newton} occasionally fails to converge within a reasonable number of iterations, but works well in many cases.
  Hence the convergence issue is not detrimental in the long run; see \Cref{tab:10_vortices_2nd_0.01_2} in \Cref{appx}.
  \smallskip
\item Broyden's method~\eqref{eq:Broyden} tends to reduce the number of required iterations slightly, and seems as fast as the simplified Newton method.
  It also resolves the convergence issue of the simplified Newton method mentioned above; see \Cref{tab:10_vortices_2nd_0.01_2} in \Cref{appx}.
\end{itemize}

In conclusion, using a small enough $\dt$ seems to be the key to an efficient implementation of our semiexplicit method.
Also, higher-order methods tend to require less iterations in the projection step.
The reason seems to be that smaller $\dt$ and higher-order methods produce smaller defects in the explicit time evolution, hence less iterations in the projection step to eliminate the defects produced.

\appendix

\section{Numerical Study of Projection Step}
\label{appx}
\subsection{Effect of Time Step and Parameters?}
This appendix gives a further numerical study of the simplified Newton method (\Cref{ssec:implementation}) and Broyden's method (\Cref{ssec:Broyden}) in the projection step.
The purpose of this appendix is to numerically investigate the following questions: (i)~How are these nonlinear solvers affected by a change in the time step as well as a change in the parameters of the problems? (ii)~Does Broyden's method offer a significant improvement of the simplified Newton method?

\subsection{NLS Problem with Larger Time Step}
\label{appx:NLS}
Let us first see how the time step affects the nonlinear solvers by changing $\dt$ from $10^{-3}$ to $10^{-2}$ in the NLS problem.

\begin{table}[htbp]
  \small
  \caption{\small
    Comparison of computation times of various methods when solving the NLS system with $N=5$ from \Cref{ssec:invariants} with time step $\dt = 10^{-2}$ and terminal time $T=10^{3}$ using both the simplified Newton~\eqref{eq:simplified_Newton} and Broyden's method~\eqref{eq:Broyden}.
    The other details are the same as \Cref{tab:Schrodinger}.
  }

 \begin{tabular}{l|rc|rc||rc|rc}
   & \multicolumn{4}{c||}{$\epsilon=10^{-10}$} & \multicolumn{4}{c}{$\epsilon=10^{-13}$} \\ \cline{2-9}
   & \multicolumn{2}{c}{Simplified~\eqref{eq:simplified_Newton}} & \multicolumn{2}{|c||}{Broyden~\eqref{eq:Broyden}} & \multicolumn{2}{c}{Simplified~\eqref{eq:simplified_Newton}} & \multicolumn{2}{|c}{Broyden~\eqref{eq:Broyden}} \\ \hline
   \textbf{method} & \texttt{Time}  & \texttt{NW\_itr} & \texttt{Time} & \texttt{NW\_itr} & \texttt{Time} & \texttt{NW\_itr} & \texttt{Time} & \texttt{NW\_itr} \\ \hline
   \textbf{Tao 2}          	& 0.76	& 	& 0.76	&   	& 0.76	& 	& 0.76	& 	\\
    \textbf{semiexplicit 2}	& 7.72	& 8.54	& 8.71	& 6.90 	& 10.27	& 11.55	& 11.04	& 8.88  \\
    \textbf{Midpoint}       	& 25.77	& 12.87	& 25.77	& 12.87	& 33.60	& 16.66	& 33.60	& 16.66 \\ \hline
    \textbf{Tao 4}          	& 1.85	& 	& 1.85	&     	& 1.85	& 	& 1.85	& 	\\
    \textbf{semiexplicit 4}  	& 10.72	& 6.79	& 11.47	& 5.78 	& 15.11	& 9.88	& 15.13	& 7.85  \\
    \textbf{semiexplicit-S 4}	& 12.53	& 5.50	& 13.07	& 4.98  & 18.95	& 8.56	& 18.46	& 6.95  \\
    \textbf{IRK4}           	& 48.60	& 10.65	& 48.60	& 10.65 & 60.50	& 13.49	& 60.50	& 13.49 \\ \hline
    \textbf{Tao 6}          	& 4.82	& 	& 4.82	&      	& 4.82	& 	& 4.82	&      	\\
    \textbf{semiexplicit 6}  	& 20.91	& 5.93	& 19.44	& 5.00  & 31.02	& 8.92	& 28.24	& 7.06  \\
    \textbf{semiexplicit-S 6}	& 15.69	& 1.78	& 16.49	& 1.78 	& 41.12	& 4.69	& 39.84	& 4.30  \\
    \textbf{semiexplicit-Y 6}	& 11.30	& 3.91	& 13.39	& 3.93  & 19.88	& 6.91	& 19.29	& 5.99  \\ \hline
  \end{tabular}
  \label{tab:Schrodinger_0.01}
\end{table}

Compare \Cref{tab:Schrodinger_0.01} ($\dt = 10^{-2}$) with \Cref{tab:Schrodinger} ($\dt = 10^{-3}$).
We observe that the change in $\dt$ results in significantly more numbers of iterations to achieve convergence in the projection step, and hence a significant slowdown of our integrator.
It also shows that, even with the increased number of iterations, Broyden's method is as fast as the simplified Newton method.
We also notice that Broyden's method tends to require less iterations.

\Cref{tab:Schrodinger_0.01_2} makes a further comparison of the average numbers of iterations for a longer time interval $0 \le t \le 10^{4}$ along with the maximum errors in the projection step and the resulting maximum defects.
Again, Broyden's method requires less iterations in the long run too, as one might expect from the theory mentioned in \Cref{ssec:Broyden}.

\begin{table}[htbp]
  \small
  \caption{\small
    Comparison of convergence and errors of various methods when solving the NLS system with $N=5$ from \Cref{ssec:invariants} with time step $\dt = 0.01$ and terminal time $T=10^{4}$ using both the simplified Newton~\eqref{eq:simplified_Newton} and Broyden's method~\eqref{eq:Broyden}. 
    \texttt{Max\_Error} denotes the maximum of $\norm{ \mu^{(N)} - \mu^{(N-1)} }$  (see \eqref{eq:proj-error}) among all projection steps (hence is approximately $\epsilon$) whereas \texttt{Max\_Defect} denotes the maximum of defects $\norm{ (q,p) - (x,y) }$ in each solution.
    The other details are the same as \Cref{tab:Schrodinger}.
  }
  \begin{tabular}{l|ccc|ccc}
    & \multicolumn{6}{c}{$\epsilon=10^{-13}$} \\ \cline{2-7}
    & \multicolumn{3}{c}{Simplified~\eqref{eq:simplified_Newton}} & \multicolumn{3}{|c}{Broyden~\eqref{eq:Broyden}} \\ \hline
    \textbf{method}		& \texttt{NW\_itr} 	& \texttt{Max\_Error}	& \texttt{Max\_Defect}       & \texttt{NW\_itr}	& \texttt{Max\_Error}	& \texttt{Max\_Defect}       \\ \hline
    \textbf{Tao 2}		& 		& 		& 0.025191	& 		& 		& 0.025191	\\
    \textbf{semiexplicit 2}		& 11.55	& 1e-13	& 4e-13	& 8.88	& 1e-13	& 4.39e-13 	\\
    \textbf{Midpoint}	& 16.66	& 1e-13	& 		& 16.66	& 1e-13	&  		\\ \hline
    \textbf{Tao 4}		& 		& 		& 0.016279	& 		& 		&  0.016279 	\\
    \textbf{semiexplicit 4}		& 9.87	& 1e-13	& 4e-13	& 7.85	& 9.99e-14	& 4.21e-13 	\\
    \textbf{semiexplicit-S 4}	& 8.56	& 1e-13	& 4e-13	& 6.95	& 1e-13	& 4.39e-13  	\\
    \textbf{IRK4}		& 13.50	& 1e-13	& 		& 13.50	& 1e-13	& 		\\ \hline
    \textbf{Tao 6}		& 		& 		& 0.006048	& 		& 		&  0.006048 	\\
    \textbf{semiexplicit 6}		& 8.93	& 9.99e-14	& 4e-13	& 7.06	& 1e-13	& 4.28e-13   	\\
    \textbf{semiexplicit-S 6}	& 4.69	& 1e-13	& 4e-13	& 4.31	& 1e-13	& 4.22e-13 	\\
    \textbf{semiexplicit-Y 6}	& 6.91	& 1e-13	& 4e-13	& 5.99	& 9.99e-14	& 4.17e-13 	\\ \hline
  \end{tabular}
  \label{tab:Schrodinger_0.01_2}
\end{table}

Notice in the \texttt{Max\_Defect} columns that the projection step suppresses the phase space defect $\norm{(q,p) - (x,y)}$ to the order of the tolerance $\epsilon = 10^{-13}$---significantly smaller than Tao's.
These defects seem to affect the errors of the solutions: \Cref{fig:Schrodinger_Inv_Tao_0.01,fig:Schrodinger_Inv_0.01} show the time evolution of two invariants (using Broyden's method) of the system with $\dt = 10^{-2}$ just as we did in \Cref{fig:Schrodinger_Inv_Tao,fig:Schrodinger_Inv}, respectively with $\dt = 10^{-3}$.
What is particularly notable is that, as $\dt$ increases, the error in the total mass $I$ increased significantly for Tao's whereas it stays very small for ours.
Therefore, our method, although our method is now much slower than Tao's, still gives much more accurate results than Tao's.

\begin{figure}[htbp]
  \centering
  \includegraphics[width=.8\linewidth]{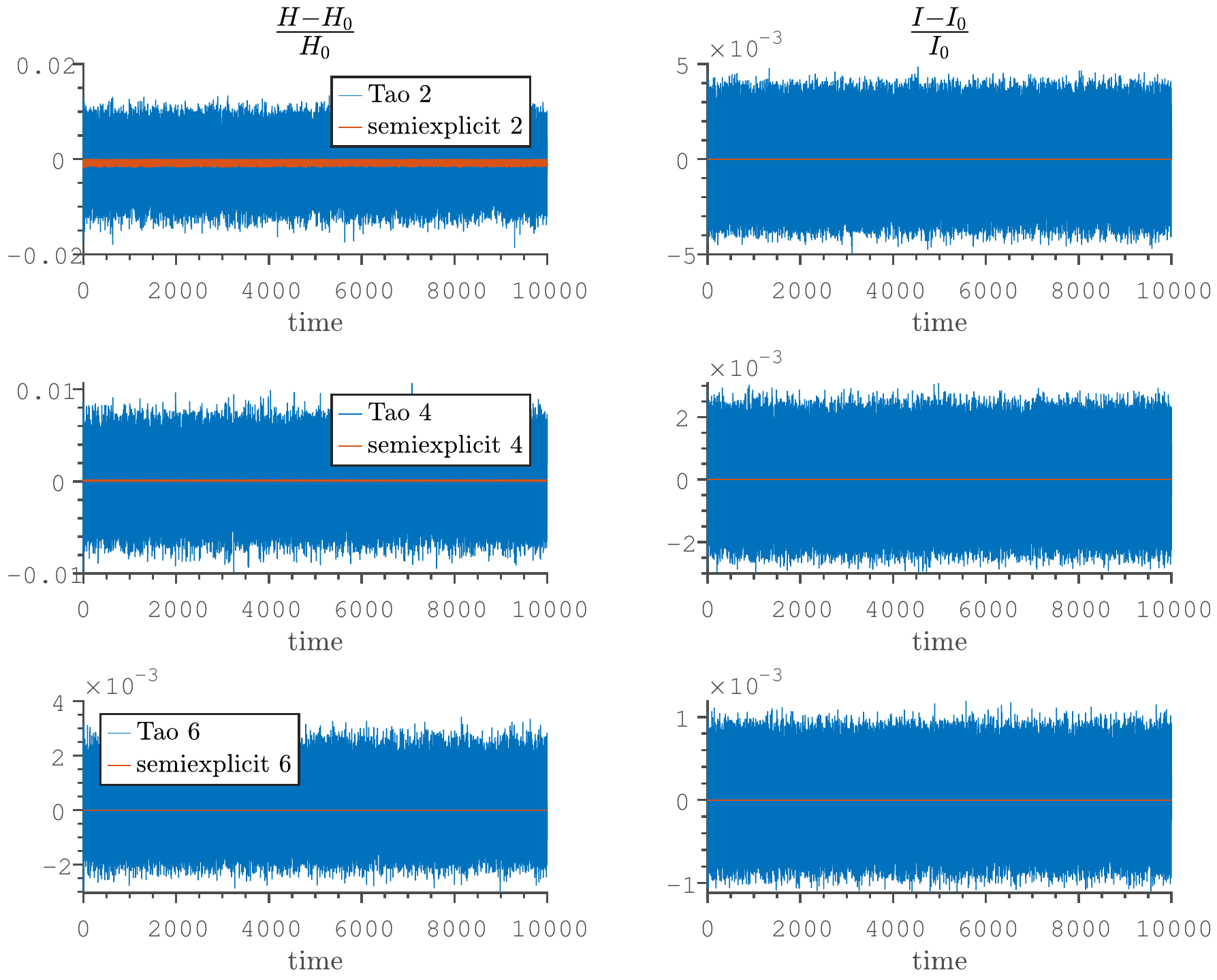}
  \caption{\small
    Same plots as \Cref{fig:Schrodinger_Inv_Tao} except with a larger time step $\dt = 10^{-2}$ instead of $\dt = 10^{-3}$.
  }
  \label{fig:Schrodinger_Inv_Tao_0.01}
\end{figure}

\begin{figure}[htbp]
  \centering
  \includegraphics[width=.8\linewidth]{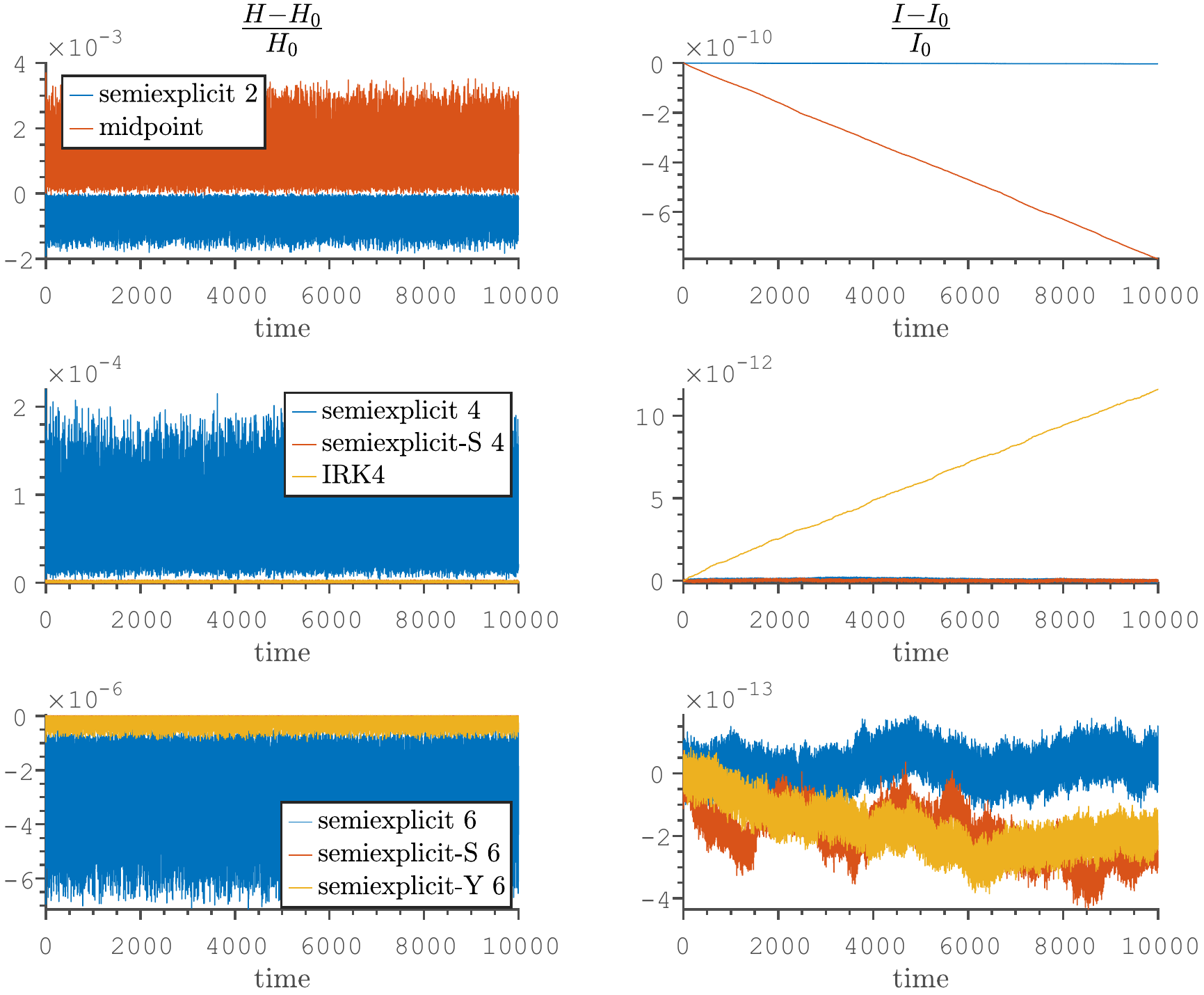}
  \caption{\small
    Same plots as \Cref{fig:Schrodinger_Inv} except with a larger time step $\dt = 10^{-2}$ instead of $\dt = 10^{-3}$.
  }
  \label{fig:Schrodinger_Inv_0.01}
\end{figure}

\subsection{10-Vortex Problem with Disparate Circulations}
\label{appx:10vortices}
Let us modify the 10-vortex problem with the following initial conditions and circulations to create a problem that requires more Newton-type iterations in the projection step.
\begin{equation}
  \label{eq:IC_10-votex_2}
  \begin{split}
    x(0)      &=(  0.5,    3.5,   -1.5,   -0.5,   -4.5,   -3.5,    1.5,   -2,    4,   -4),\\
    y(0)      &=(  5,    0.5,    2,    5,   -2,   -1,   -0.5,    3,    3.5,   -4),\\
    \Gamma    &=(-14.8,  -18.8,   17.6,   -8,   -8.2,   -6.8,   -1.4,    6,  -11,   13.8).
  \end{split}
\end{equation}
Compare this to the previous parameters~\eqref{eq:IC_10-votex} where $-1 < \Gamma_{i} < 1$ for $i = 1, \dots, 10$.
The circulations now have more disparities among the vortices.
This means that vortices move in different time scales, making the problem more challenging in selecting the time step.
Accordingly, we have chosen a smaller time step $\dt = 0.01$ here.

\begin{table}[htbp]
  \small
  \caption{\small
    Comparison of computation times of various methods when solving the 10-vortex system with parameters~\eqref{eq:IC_10-votex_2} in place of \eqref{eq:IC_10-votex} and time step $\dt = 0.01$ and terminal time $T=100$ using both the simplified Newton~\eqref{eq:simplified_Newton} and Broyden's method~\eqref{eq:Broyden}.
    The other details are the same as \Cref{tab:10_vortices}.
  }
  \begin{tabular}{l|rc|rc||rc|rc}
    & \multicolumn{4}{c||}{$\epsilon=10^{-10}$} & \multicolumn{4}{c}{$\epsilon=10^{-13}$} \\ \cline{2-9}
    & \multicolumn{2}{c}{Simplified~\eqref{eq:simplified_Newton}} & \multicolumn{2}{|c||}{Broyden~\eqref{eq:Broyden}} & \multicolumn{2}{c}{Simplified~\eqref{eq:simplified_Newton}} & \multicolumn{2}{|c}{Broyden~\eqref{eq:Broyden}} \\ \hline
    \textbf{method} & \texttt{Time}  & \texttt{NW\_itr} & \texttt{Time} & \texttt{NW\_itr} & \texttt{Time} & \texttt{NW\_itr} & \texttt{Time} & \texttt{NW\_itr} \\ \hline
    \textbf{Tao 2}          	& 2.05	& 	& 2.05	& 	& 2.05	& 	& 2.05	&      \\
    \textbf{semiexplicit 2}	& 5.04	& 3	& 5.01	& 3 	& 8.12	& 5	& 7.68	& 4.55 \\
    \textbf{Midpoint}       	& 24.50	& 4.62	& 24.50	& 4.62	& 31.12	& 5.77	& 31.12	& 5.77 \\ \hline
    \textbf{Tao 4}          	& 5.96	& 	& 5.96	& 	& 5.96	& 	& 5.96	&      \\
    \textbf{semiexplicit 4}  	& 8.44	& 1.85	& 8.89	& 1.85	& 13.96	& 3	& 14.05	& 3    \\
    \textbf{semiexplicit-S 4}	& 8.08	& 1.08	& 8.43	& 1.08	& 17.19	& 2.29	& 17.86	& 2.28 \\
    \textbf{IRK4}           	& 58.71	& 4.30	& 58.71	& 4.30	& 73.88	& 5.25	& 73.88	& 5.25 \\ \hline
    \textbf{Tao 6}          	& 17.67	& 	& 17.67	& 	& 17.67	& 	& 17.67	&      \\
    \textbf{semiexplicit 6}  	& 13.35	& 1.00	& 13.61	& 1.00	& 24.36	& 1.80	& 25.28	& 1.80 \\
    \textbf{semiexplicit-S 6}	& 37.15	& 1.00	& 37.15	& 1.00 	& 40.60	& 1.08	& 40.58	& 1.08 \\
    \textbf{semiexplicit-Y 6}	& 10.51	& 1	& 10.42	& 1     & 11.01	& 1.03	& 11.09	& 1.03 \\ \hline
  \end{tabular}

  \label{tab:10_vortices_2nd_0.01}
\end{table}

\Cref{tab:10_vortices_2nd_0.01} shows the average computation times as well as the average number of iterations for this problem just as in \Cref{tab:10_vortices} for $0 \le t \le 100$.
We observe an increase in the average number of iterations and hence a slight slowdown of our method.
As a result, our method is slightly slower than Tao's with $\epsilon = 10^{-13}$, although with $\epsilon= 10^{-10}$, the our $6^{\rm th}$-order method is still faster than Tao's.

\begin{table}[htbp]
  \small
  \caption{\small
    Comparison of convergence and errors of various methods when solving the challenging 10-vortex system with time step $\dt = 0.01$ and a longer terminal time $T=1000$ using both the simplified Newton~\eqref{eq:simplified_Newton} and Broyden's method~\eqref{eq:Broyden}.
    \texttt{Max\_Error} denotes the maximum of $\norm{ \mu^{(N)} - \mu^{(N-1)} }$  (see \eqref{eq:proj-error}) among all projection steps (hence is approximately $\epsilon$) whereas \texttt{Max\_Defect} denotes the maximum of defects $\norm{ (q,p) - (x,y) }$ in each solution.
    The star $^{*}$ indicates that the Newton iteration was stopped at $N = 100$ before it converged, hence resulting in \texttt{Max\_Error} larger than $\epsilon$.
    The other details are the same as \Cref{tab:10_vortices}.
  }
  \begin{tabular}{l|ccc|ccc}
    & \multicolumn{6}{c}{$\epsilon=10^{-13}$} \\ \cline{2-7}
    & \multicolumn{3}{c}{Simplified~\eqref{eq:simplified_Newton}} & \multicolumn{3}{|c}{Broyden~\eqref{eq:Broyden}} \\ \hline
    \textbf{method}		& \texttt{NW\_itr} 	& \texttt{Max\_Error}	& \texttt{Max\_Defect}       & \texttt{NW\_itr}	& \texttt{Max\_Error}	& \texttt{Max\_Defect}       \\ \hline
    \textbf{Tao 2}			& 			& 		& 0.00376	& 			& 		& 0.00376    	\\
    \textbf{semiexplicit 2}		& 5.27$^{*}$	& 1.67e-13	& 6.66e-13	& 4.46$\:\:$		& 9.99e-14	& 6.57e-13	\\
    \textbf{Midpoint}		& 5.71$\:\:$		& 9.99e-14	& 		& 5.71$\:\:$		& 9.99e-14	& 		\\ \hline
    \textbf{Tao 4}			& 			& 		& 8.38e-06	& 			& 		& 8.38e-06	\\
    \textbf{semiexplicit 4}		& 3.59$^{*}$	& 1.62e-13	& 6.47e-13	& 3.05$\:\:$		& 9.99e-14	& 7.32e-13	\\
    \textbf{semiexplicit-S 4}	& 2.64$^{*}$	& 1.61e-13	& 6.44e-13	& 2.45$\:\:$		& 1e-13	& 7.11e-13 	\\
    \textbf{IRK4}			& 5.58$^{*}$	& 1.18e-13	& 		& 5.58$^{*}$	& 1.18e-13	& 		\\ \hline
    \textbf{Tao 6}			& 			& 		& 2.09e-07	& 			& 		& 2.09e-07	\\
    \textbf{semiexplicit 6}		& 1.97$^{*}$	& 1.42e-13	& 5.69e-13	& 1.90$\:\:$		& 9.99e-14	& 7.15e-13   	\\
    \textbf{semiexplicit-S 6}	& 1.40$^{*}$	& 1.16e-13	& 4.62e-13	& 1.41$\:\:$		& 9.99e-14	& 7.84e-13	\\
    \textbf{semiexplicit-Y 6}	& 1.14$\:\:$		& 1e-13	& 4e-13	& 1.14$\:\:$		& 9.99e-14	& 5.41e-13	\\ \hline
  \end{tabular}

  \label{tab:10_vortices_2nd_0.01_2}
\end{table}

\Cref{tab:10_vortices_2nd_0.01_2} shows the average numbers of iterations along with the maximum errors in the projection step and the resulting maximum defects computed for the longer interval $0 \le t \le 1000$.
Overall, there is a slight increase in the average number of iterations for every method.
Particularly, notice those numbers marked with star $^{*}$ in the \texttt{NW\_itr} first column.
These are the instances where the simplified Newton method did not converge to the desired value $\mu$ with $\epsilon=10^{-13}$ within 100 steps at some point in the simulation.
However, the only slight increases in \texttt{NW\_itr} compared to those in \Cref{tab:10_vortices_2nd_0.01} suggest that these are rather rare instances.
Notice also that this issue of convergence is resolved by using Broyden's method.

\Cref{fig:vortex_Inv_Tao_2nd,fig:vortex_Inv_2nd} show the time evolution of the invariants (using Broyden's method) just as we did in \Cref{fig:vortex_Inv_Tao,fig:vortex_Inv} with different parameters.
Notice that, although Tao's method is faster in most cases, its errors in some invariants increased whereas ours either stayed in the same level or became smaller.
It is also worth mentioning that, the last row of \Cref{fig:vortex_Inv_Tao} is the instance where our method is faster than Tao's, and ours is much more accurate than Tao's.

\begin{figure}[htbp]
  \centering
  \includegraphics[width=.9\linewidth]{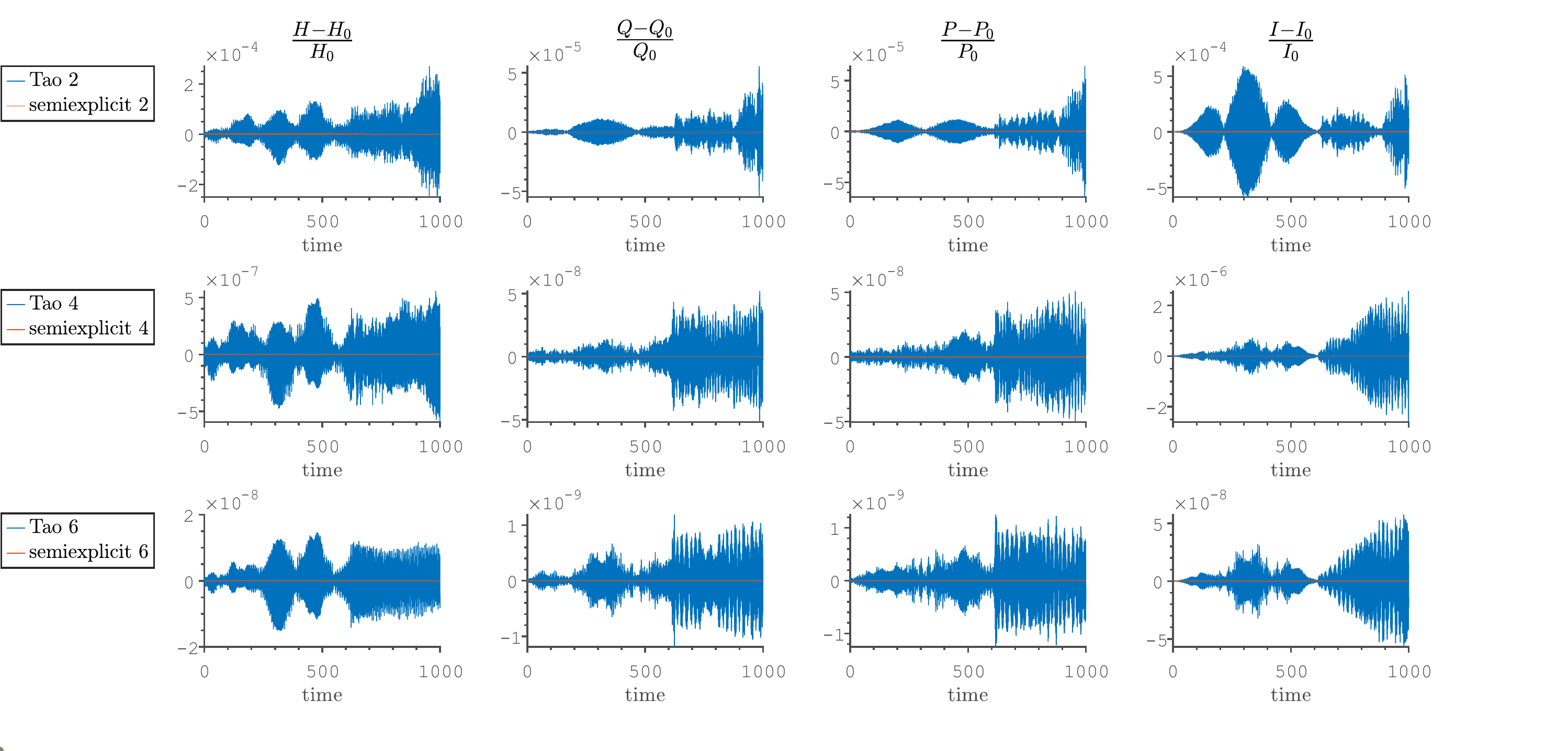}
  \caption{\small
    Same plots as \Cref{fig:vortex_Inv_Tao} except with parameters \eqref{eq:IC_10-votex_2} and $\dt = 0.01$ instead of \eqref{eq:IC_10-votex} and $\dt = 0.1$.
  }
  \label{fig:vortex_Inv_Tao_2nd}
\end{figure}

\begin{figure}[htbp]
  \centering
  \includegraphics[width=.9\linewidth]{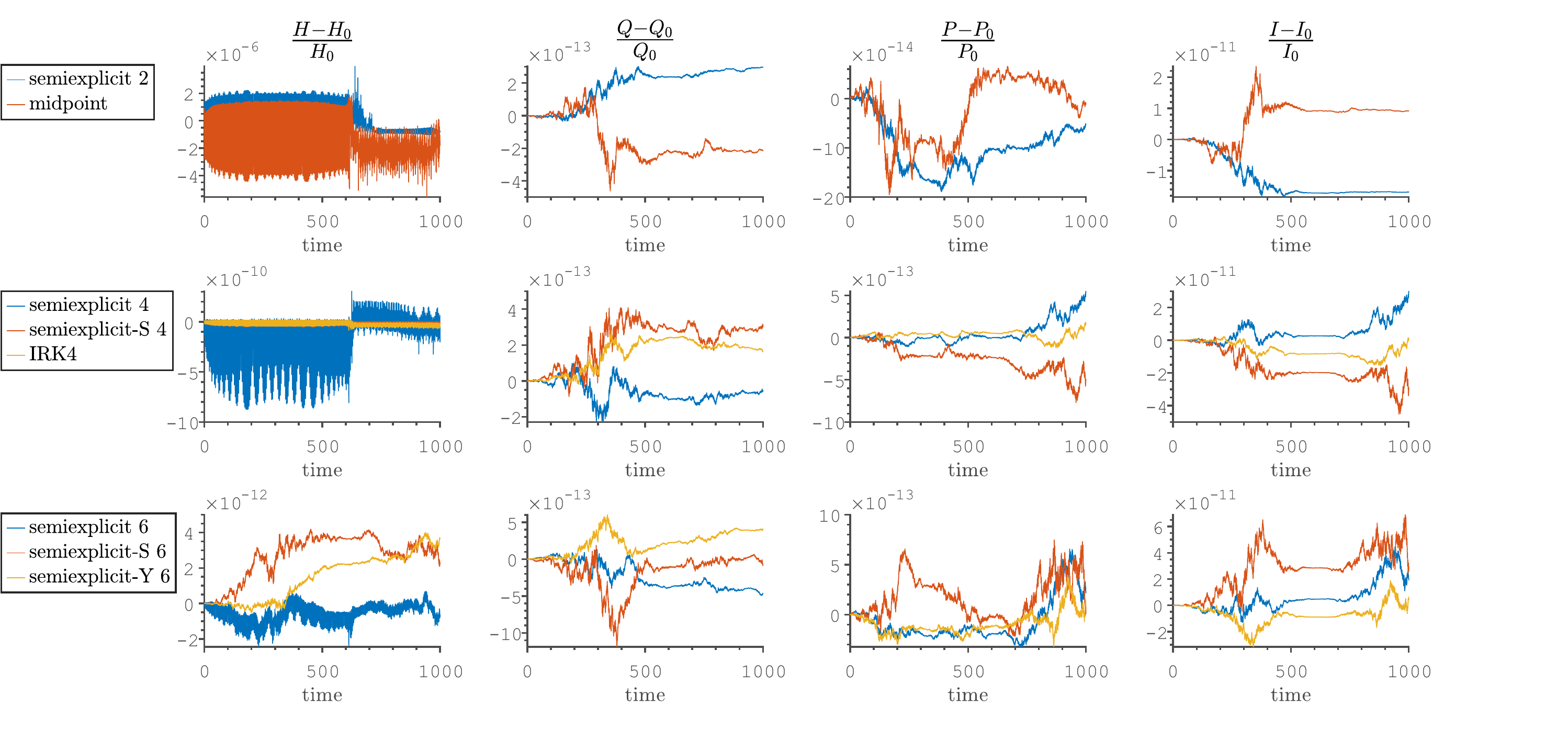}
  \caption{\small
    Same plots as \Cref{fig:vortex_Inv} except with parameters \eqref{eq:IC_10-votex_2} and $\dt = 0.01$ instead of \eqref{eq:IC_10-votex} and $\dt = 0.1$.
  }
  \label{fig:vortex_Inv_2nd}
\end{figure}

\subsection{Conclusion of Numerical Study}
The above numerical experiments suggest first that the time step and the parameters of the problem affect the number of iterations in the projection step considerably.
Particularly, the NLS example in \Cref{appx:NLS} shows that increasing the time step may result in a significant increase in the number of iterations.
The 10-vortex example in \Cref{appx:10vortices} shows that the change in the parameters also affects the convergence of the projection step too, but can be compensated by selecting an appropriate time step $\dt$, again underscoring the importance of selecting a small enough $\dt$.

The results also show that Broyden's method is as fast as the simplified Newton method, and is more robust in the sense that it does not experience the convergence issue that the simplified Newton method occasionally did.
However, the convergence issue is rather minor, and the overall difference between the two methods is fairly small, suggesting that the simplified Newton method is quite effective despite its simplicity.

\section*{Acknowledgments}
We would like to thank Molei Tao for helpful discussions, and Joshua Burby, Michael Kraus, and Xin Wu for insightful comments.
We also would like to thank the reviewer for the questions and comments that helped us improve the paper.
This work was supported by NSF grant DMS-2006736.

\subsection*{Note added in Proof}
The numerical results suggest that our semiexplicit integrator preserves both the total mass $I$ (see \eqref{eq:total_mass} of the NLS problem and the angular impulse $I$ (see \eqref{eq:impulses}) of the vortex problem---both quadratic invariants.
In fact, our semiexplicit integrator preserves any quadratic invariant of the original Hamiltonian system~\eqref{eq:Ham}.
We shall report on this in the forthcoming paper~\cite{Oh-Semiexplicit2}.

\bibliographystyle{plainnat}
\bibliography{Semiexplicit}

\end{document}